\documentclass{article}%
\usepackage{amsmath}
\usepackage{amssymb}
\usepackage{amsfonts}
\usepackage{graphicx}%
\providecommand{\U}[1]{\protect\rule{.1in}{.1in}}
\newtheorem{theorem}{Theorem}[section]

\newtheorem{corollary}[theorem]{Corollary}

\newtheorem{definition}[theorem]{Definition}

\newtheorem{lemma}[theorem]{Lemma}

\newtheorem{proposition}[theorem]{Proposition}

\newenvironment{proof}[1][Proof]{\noindent\textbf{#1.} }{\ \rule{0.5em}{0.5em}}
\begin{document}

\title{On graded irreducible representations of Leavitt path algebras}
\author{Roozbeh Hazrat$^{(a)}$ and Kulumani M. Rangaswamy$^{(b)}$\\(a) Centre for Research in Mathematics, \\University of Western Sydney, Sydney, Australia\\E-mail: r.hazrat@uws.edu.au\\(b) Department of Mathematics,\\University of Colorado at Colorado Springs,\\Colorado Springs, Colorado 80918, USA\\E-mail: krangasw@uccs.edu}
\date{}
\maketitle

\begin{abstract}
Using the $E$-algebraic branching systems,\ various graded irreducible
representations of a Leavitt path $K$-algebra $L$ of a directed graph $E$ are
constructed. The concept of a Laurent vertex is introduced and it is shown
that the minimal graded left ideals of $L$ are generated by the Laurent
vertices or the line points leading to a detailed description of the graded
socle of $L$. Following this, a complete characterization is obtained of the
Leavitt path algebras over which every graded irreducible representation is
finitely presented. A useful result is that the irreducble representation
$V_{[p]}$ induced by infinite paths tail-equivalent to an infinite path $p$
(we call this a Chen simple module) is graded if and only if $p$ is an
irrational path. We also show that every one-sided ideal of $L$ is graded if
and only if the graph $E$ contains no cycles. Supplementing the theorem of one
of the co-authors that every Leavitt path algebra $L$ is graded von Neumann
regular, we show that $L$ is graded self-injective if and only if $L$ is a
graded semi-simple algebra, made up of matrix rings of arbitrary size over the
field $K$ or the\ graded field $K[x^{n},x^{-n}]$ where $n\in%
\mathbb{N}
$.

\end{abstract}

\section{Introduction and Preliminaries}

As mentioned in \cite{GR-1}, branching systems arise in the study of various
disciplines such as random walk, symbolic dynamics and scientific computing
(see e.g. \cite{D}, \cite{H}, \cite{SS}). Given an arbitrary directed graph
$E$, Gon\c{c}alves and Royer defined in \ \cite{GR-1} and \cite{GR-0} a
branching system using a measure space $(X,\mu)$ and indicated a method of
constructing a large number of representations of the graph $C^{\ast}$-algebra
$C^{\ast}(E)$ in the space of bounded linear operators on\ $L^{2}(X,\mu)$. As
an algebraic analogue of\ this theory, they defined the concept of an
$E$-algebraic branching system in \cite{GR} to construct various
representations of Leavitt path algebras (which are algebraic analogues of
graph C*-algebras). Following this, X.W.Chen (\cite{C}) constructed
irreducible representations of a Leavitt path algebra $L$ of a graph $E$ using
the sinks as well as infinite paths in $E$ and showed that these
representations are also induced by appropriate $E$-algebraic branching
systems. Just about the same time, the relations between the theory of quiver
representations and the theory of representations of Leavitt path algebras
were explored in \cite{AB}. Following this there was a flurry of activity in
investigating various irreducible representations of $L$ induced by infinite
emitters and cycles in $E$, leading to the study of Leavitt path algebras with
a special type of or a specific number of irreducible representations (see
\cite{AMT}, \cite{ARa}, \cite{R-2}, \cite{R-3}).

Leavitt path algebras are also endowed with a compatible $%
\mathbb{Z}
$-grading and the initial study of the graded structure of these algebras was
done in \cite{H-1}. Moreover, it was shown in \cite{H-2} that every Leavitt
path algebra is a graded von Neumann regular algebra, an indication of a
possible rich internal structure of these algebras. As a first step in
understanding the graded modules over these algebras, we consider in this
paper the graded irreducible representations of a Leavitt path algebra $L$ of
an arbitrary graph $E$. It is known for a graded von Neumann regular ring $R$
with identity, that there is a one-to-one correspondence between the graded
minimal left ideals of $R$ and the minimal left ideals of $R_{0}$. When $R$ is
a Leavitt path algebra, $R_{0}$ is ultramatricial. While it may not always be
possible to describe all the minimal left ideals of $R_{0}$, our graphical
approach enables us to provide a complete description of the graded minimal
left/right ideals of $R$. Defining a grade on the $E$-algebraic branching
systems, we construct various types of graded irreducible representations of
$L$. A useful result is that if $p$ is an infinite path, then the irreducible
representation $V_{[p]}$ defined by the equivalence class of infinite paths
tail-equivalent to $p$ is graded if and only if $p$ is an irrational path (see
definitions below). We also describe the annihilator ideals of many of the
graded irreducible representations. It turns out that, unlike the non-graded
case, the annihilator of a graded irreducible representation need not be a
primitive ideal, but is a graded prime ideal of $L$. The concept of a Laurent
vertex (see Definition \ref{Defn Laurent vertex} below) is introduced and is
shown that a graded left ideal of $L$ is a minimal graded left ideal if and
only if it is isomorphic to the left ideal $Lv$ with possibly a shifted
grading, where $\ v$ is either a Laurent vertex or a line point. This enables
us to describe the graded socle of $L$ as the ideal generated by these special
vertices in $E$ and that it is a graded direct sum of matrix rings of
arbitrary size over $K$ and/or $K[x,x^{-1}]$ with appropriate gradings.

Next we investigate when a graded irreducible representation of $L$ is
finitely presented. A key result is that a graded irreducible representation
$S$ induced by an infinite path $p$ is finitely presented if and only if $p$
contains a line point and in that case $S$ becomes graded projective. Using
this, it is shown that every irreducible representation of $L$ is finitely
presented if and only if the graph $E$ is row-finite and that every path in
$E$ eventually ends at a line point or a Laurent vertex. In this case, the
structure of $L$ is also completely described.

A well-known theorem in Leavitt path algebras (\cite{T}) states that every
two-sided ideal of a Leavitt path algebra $L:=L_{K}(E)$ is graded if and only
if the graph $E$ satisfies Condition (K) (see definition below). A natural
question is when will every left (right) ideal of $L$ be graded. We show that
this happens exactly when the graph $E$ contains no cycles. In this case, $L$
is a directed union of graded subalgebras $B$ each of which is a graded direct
sum of finitely many matrix rings of finite order over $K$ endowed with
natural gradings.

It is known that Leavitt path algebras which are self-injective are always von
Neumann regular (see \cite{ARS}). Since by \cite{H-2} every Leavitt path
algebra is graded von Neumann regular, it is natural to consider the subclass
of Leavitt path algebras which are graded self-injective. We show that such
algebras are just the graded semi-simple $K$-algebras. The corresponding graph
theoretical properties of the graph $E$ are also described.

As an illustration of the advantage of viewing a Leavitt path algebra $L$ as a
$%
\mathbb{Z}
$-graded algebra, we show how the properties of $L$ as a graded structure
enables us to simplify, generalize or correct the proofs of some of the
important previously published results on Leavitt path algebras (see Theorems
4.2, 4.6 and Proposition 6.2).

\bigskip

For the general notation, terminology and results in Leavitt path algebras, we
refer to \cite{AA} and \cite{R-1}. We give below a short outline of some of
the needed basic concepts and results.

A (directed) graph $E=(E^{0},E^{1},r,s)$ consists of two sets $E^{0}$ and
$E^{1}$ together with maps $r,s:E^{1}\rightarrow E^{0}$. The elements of
$E^{0}$ are called \textit{vertices} and the elements of $E^{1}$
\textit{edges}. All the graphs $E$ that we consider (excepting when
specifically stated) are arbitrary in the sense that no restriction is placed
either on the number of vertices in $E$ or on the number of edges emitted by a
single vertex. In this connection, we wish to point out that some of the
earlier papers in Leavitt path algebras, such as \cite{AAPS}, \cite{AMMS}
which we quote in this paper, assume that the graphs considered are countable.
But the proofs of the results in these papers are valid even for uncountable graphs.

A vertex $v$ is called a \textit{sink} if it emits no edges and a vertex $v$
is called a \textit{regular} \textit{vertex} if it emits a non-empty finite
set of edges. An \textit{infinite emitter} is a vertex which emits infinitely
many edges. A graph without infinite emitters is said to be row-finite. For
each $e\in E^{1}$, we call $e^{\ast}$ a ghost edge. We let $r(e^{\ast})$
denote $s(e)$, and we let $s(e^{\ast})$ denote $r(e)$. A\textit{ path} $\mu$
of length $n>0$ is a finite sequence of edges $\mu=e_{1}e_{2}\cdot\cdot\cdot
e_{n}$ with $r(e_{i})=s(e_{i+1})$ for all $i=1,\cdot\cdot\cdot,n-1$. In this
case $\mu^{\ast}=e_{n}^{\ast}\cdot\cdot\cdot e_{2}^{\ast}e_{1}^{\ast}$ is the
corresponding ghost path. A vertex is considered a path of length $0$. The set
of all vertices on the path $\mu$ is denoted by $\mu^{0}$. We shall denote the
length of a path $\mu$ by $|\mu|$.

A path $\mu$ $=e\cdot\cdot\cdot e_{n}$ in $E$ is \textit{closed} if
$r(e_{n})=s(e_{1})$, in which case $\mu$ is said to be based at the vertex
$s(e_{1})$. A closed path $\mu$ as above is called \textit{simple} provided it
does not pass through its base more than once, i.e., $s(e_{i})\neq s(e_{1})$
for all $i=2,...,n$. The closed path $\mu$ is called a \textit{cycle} if it
does not pass through any of its vertices twice, that is, if $s(e_{i})\neq
s(e_{j})$ for every $i\neq j$. An \textit{exit} for a path $\mu$ $=e_{1}%
\cdot\cdot\cdot e_{n}$ is an edge $f$ that satisfies $s(f)=s(e_{i})$ for some
$i$ but $f\neq e_{i}$. The graph $E$ is said to satisfy \textit{Condition
(L)}, if every cycle in $E$ has an exit in $E$. The graph $E$ is said to
satisfy\textit{ Condition (K),} if any vertex on a closed path $\mu$ is also
the base for a closed path $\gamma$ different from $\mu$.

If there is a path from vertex $u$ to a vertex $v$, we write $u\geq v$. A
subset $D$ of vertices is said to be \textit{downward directed }\ if for any
$u,v\in D$, there exists a $w\in D$ such that $u\geq w$ and $v\geq w$. A
subset $H$ of $E^{0}$ is called \textit{hereditary} if, whenever $v\in H$ and
$w\in E^{0}$ satisfy $v\geq w$, then $w\in H$. A hereditary set is
\textit{saturated} if, for any regular vertex $v$, $r(s^{-1}(v))\subseteq H$
implies $v\in H$.

Given an arbitrary graph $E$ and a field $K$, the \textit{Leavitt path algebra
}$L_{K}(E)$ is defined to be the $K$-algebra generated by a set $\{v:v\in
E^{0}\}$ of pairwise orthogonal idempotents together with a set of variables
$\{e,e^{\ast}:e\in E^{1}\}$ which satisfy the following conditions:

(1) \ $s(e)e=e=er(e)$ for all $e\in E^{1}$.

(2) $r(e)e^{\ast}=e^{\ast}=e^{\ast}s(e)$\ for all $e\in E^{1}$.

(3) (The "CK-1 relations") For all $e,f\in E^{1}$, $e^{\ast}e=r(e)$ and
$e^{\ast}f=0$ if $e\neq f$.

(4) (The "CK-2 relations") For every regular vertex $v\in E^{0}$,
\[
v=\sum_{e\in E^{1},s(e)=v}ee^{\ast}.
\]

For any vertex $v$, the \textit{tree} of $v$ is $T_{E}(v)=\{w\in E^{0}:v\geq
w\}$. We say there is a\textit{ bifurcation} at a vertex $v$ or $v$ is
a\textit{ bifurcation vertex}, if $v$ emits more than one edge. In a graph
$E$, a vertex $v$ is called a \textit{line point} if there is no bifurcation
or a cycle based at any vertex in $T_{E}(v)$. Thus, if $v$ is a line point,
there will be a single finite or infinite line segment $\mu$ starting at $v$
($\mu$ could just be $v$) and any other path $\alpha$ with $s(\alpha)=v$ will
just be an initial subpath of $\mu$. It was shown in \cite{AMMS} that $v$ is a
line point in $E$ if and only if $L_{K}(E)v$ (and likewise $vL_{K}(E)$) is a
simple left (right) ideal of $L$.

We shall be using the following concepts and results from \cite{T}. A
\textit{breaking vertex }of a hereditary saturated subset $H$ is an infinite
emitter $w\in E^{0}\backslash H$ with the property that $0<|s^{-1}(w)\cap
r^{-1}(E^{0}\backslash H)|<\infty$. The set of all breaking vertices of $H$ is
denoted by $B_{H}$. For any $v\in B_{H}$, $v^{H}$ denotes the element
$v-\sum_{s(e)=v,r(e)\notin H}ee^{\ast}$. Given a hereditary saturated subset
$H$ and a subset $S\subseteq B_{H}$, $(H,S)$ is called an \textit{admissible
pair.} Given an admissible pair $(H,S)$, the ideal generated by $H\cup
\{v^{H}:v\in S\}$ is denoted by $I(H,S)$. It was shown in \cite{T} that the
graded ideals of $L_{K}(E)$ are precisely the ideals of the form $I(H,S)$ for
some admissible pair $(H,S)$. Moreover, $L_{K}(E)/I(H,S)\cong L_{K}%
(E\backslash(H,S))$. Here $E\backslash(H,S)$ is the \textit{quotient graph of
}$E$ in which\textit{ }%
\[
(E\backslash(H,S))^{0}=(E^{0}\backslash H)\cup\{v^{\prime}:v\in B_{H}%
\backslash S\}
\]
and
\[
(E\backslash(H,S))^{1}=\{e\in E^{1}:r(e)\notin H\}\cup\{e^{\prime}:e\in
E^{1},r(e)\in B_{H}\backslash S\}.
\]
Further, $r$ and $s$ are extended to $(E\backslash(H,S))^{1}$ by setting
$s(e^{\prime})=s(e)$ and $r(e^{\prime})=r(e)^{\prime}$.

A useful observation is that every element $a$ of $L_{K}(E)$ can be written as
$a=%
{\textstyle\sum\limits_{i=1}^{n}}
k_{i}\alpha_{i}\beta_{i}^{\ast}$, where $k_{i}\in K$, $\alpha_{i},\beta_{i}$
are paths in $E$ and $n$ is a suitable integer. Moreover, $L_{K}(E)=%
{\displaystyle\bigoplus\limits_{v\in E^{0}}}
L_{K}(E)v=%
{\displaystyle\bigoplus\limits_{v\in E^{0}}}
vL_{K}(E)$ (see \cite{AA}).

Recall that a $K$-algebra $R$ is said to be a $%
\mathbb{Z}
$\textit{-graded algebra} if $R=%
{\displaystyle\bigoplus\limits_{n\in\mathbb{Z}}}
R_{n}$ is a direct sum of $K$-vector spaces $R_{n}$, which satisfy the
property that $R_{m}R_{n}\subseteq R_{m+n}$ for all $m,n$. Elements of the
subspaces $R_{n}$ are called \textit{homogeneous elements}. A left $R$-module
$M$ is called a\textit{ graded left }$R$\textit{-module}, if $M=%
{\displaystyle\bigoplus\limits_{n\in\mathbb{Z}}}
M_{n}$ is a direct sum of subgroups $M_{n}$ which satisfy the property that
$R_{m}M_{n}\subseteq M_{m+n}$ for all $m,n$. For any $k\in%
\mathbb{Z}
$, the\textit{ }$k$\textit{-shifted module} $M(k)$ is the graded $R$-module
for which the homogeneous component $M(k)_{n}=M_{k+n}$ for all $n$. \ A graded
ring $R$ with $1$ is called a \textit{graded division ring }if every non-zero
homogeneous element in $R$ is invertible. A commutative graded division ring
is called a \textit{graded field}. The graded fields that we encounter in this
paper are $%
\mathbb{Z}
$-graded subrings of $K[x,x^{-1}]$ of the form $K[x^{n},x^{-n}]$ \ with
support $n%
\mathbb{Z}
$, where $n\in%
\mathbb{N}
$. For the general properties of graded rings and graded modules, we refer to
\cite{NvO}. Every Leavitt path algebra $L:=L_{K}(E)$ is a $%
\mathbb{Z}
$-graded algebra $L=%
{\displaystyle\bigoplus\limits_{n\in\mathbb{Z}}}
L_{n}$ induced by defining, for all $v\in E^{0}$ and $e\in E^{1}$, $\deg
(v)=0$, $\deg(e)=1$, $\deg(e^{\ast})=-1$. Further, for each $n\in%
\mathbb{Z}
$, the homogeneous component $L_{n}$ is given by
\[
L_{n}=\{%
{\textstyle\sum}
k_{i}\alpha_{i}\beta_{i}^{\ast}\in L:\text{ }|\alpha_{i}|-|\beta_{i}|=n\}.
\]
If $p=e_{1}\cdot\cdot\cdot e_{n}e_{n+1}\cdot\cdot\cdot$ \ \ \ is an infinite
path in $E$, we follow Chen \cite{C} to define, for each $n\geq1$, $\tau^{\leq
n}(p)=e_{1}\cdot\cdot\cdot e_{n}$ \ and $\tau^{>n}(p)=e_{n+1}e_{n+2}\cdot
\cdot\cdot$ \ . Two infinite paths $p,q$ are said to be
\textit{tail-equivalent} if there are positive integers $m,n$ such that
$\tau^{>m}(p)=\tau^{>n}(q)$. This defines an equivalence relation among the
infinite paths in $E$ and the equivalence class containing the path $p$ is
denoted by $[p]$. An infinite path $p$ is said to a \textit{rational path }if
it is tail equivalent to an infinite path $q=ccc\cdot\cdot\cdot$ , where $c$
is a closed path. An infinite path which is not rational is called an
\textit{irrational path}.

\bigskip

\textbf{Convention}: In what follows $E$ denotes an arbitrary graph, $K$
denotes a field and $L$ denotes the Leavitt path algebra $L_{K}(E).$ Also, all
the $L$-modules are assumed to be left $L$-modules.

\section{Minimal Graded Left Ideals and the Graded Socle}

The minimal left (right) ideals of Leavitt path algebras of arbitrary graphs
were determined in \cite{AMMS}. It was proved that $Lv$ is a minimal left
ideal if and only if $v$ is a line point. This can follow from a more general
statement in the settings of semiprime rings $R$. Namely, for an idempotent
$\epsilon$, $R\epsilon$ is a minimal left ideal if and only if $\epsilon
R\epsilon$ is a division ring. In this section, we will use the graded version
of the above statement to determine the graded minimal left (right) ideals.
Specifically, we introduce the concept of a Laurent vertex in a graph and show
that,\ in a Leavitt path algebra $L$, every minimal graded left ideal is
isomorphic to $Lv(n)$ where $v$ is either a Laurent vertex or a line point and
$n\in%
\mathbb{Z}
$. Using this, we obtain our structure theorem\ for the graded socle
$Soc^{gr}(L)$ of a Leavitt path algebra $L$ showing that $Soc^{gr}(L)$ is a
ring direct sum of matrix rings of arbitrary size over $K$ and/or over
$K[x^{t},x^{-t}]$ with appropriate grading, where $t\in%
\mathbb{N}
$.

Recall that a graded $L$-module $M$ is said to be graded-simple, if $0$ and
$M$ are the only graded submodules of $M$. Observe that a graded-simple
$L$-module need not be simple. For example, $M=K[x,x^{-1}]$ is graded-simple
as a $K[x,x^{-1}]$-module. But it is not simple, since there are infinitely
many ideals of $K[x,x^{-1}]$ giving rise to proper submodules of $M$. On the
other hand, a simple module need not be graded-simple. Indeed, if $I$ is the
ideal generated by any irreducible polynomial $p(x)\in K[x,x^{-1}]$, then
$K[x,x^{-1}]/I$ is a simple module over $K[x,x^{-1}]$, but it is not
graded-simple, as it is not even a graded module over $K[x,x^{-1}]$. \ We
shall see in the sequel more occurrences of such examples.

We first start with a few preparatory lemmas.

\begin{lemma}
\label{suspension graded simple}Let $I$ be a graded left ideal of a $%
\mathbb{Z}
$-graded ring $R$. Then $I$ is a minimal graded left ideal if and only if
$I(n)$ is a minimal graded\ left $R$-module for any $n\in%
\mathbb{Z}
$.
\end{lemma}

The proof is left to the reader.

\begin{lemma}
\label{No Bifurcation imples iso} If a vertex $v$ is connected to a vertex $u$
by a path $p$ with no bifurcations, then $Lv(-|p|)\cong_{gr}Lu$ as graded left modules.
\end{lemma}

\begin{proof}
Since $pp^{\ast}=v$ and $p^{\ast}p=u$, the map $av\longmapsto ap$ is a graded
isomorphism from $Lv(-|p|)$ to $Lu$.
\end{proof}

The next Lemma was proved for non-graded simple left ideals in \cite{AMMS}.

\begin{lemma}
\label{Simple implies Tree has no bifurcation} If, for a vertex $v$, $Lv$ is a
minimal graded left ideal of $L$, then $T_{E}(v)$ \ does not contain any
bifurcating vertices.
\end{lemma}

\begin{proof}
Indeed if $T_{E}(v)$ contains bifurcating vertices, choose a bifurcating
vertex $u$ such that there are no bifurcations in the path $p$ connecting $v$
to $u$. By Lemma \ref{No Bifurcation imples iso}, $Lu\cong Lv(-|p|)$ and by
Lemma \ref{suspension graded simple}, $Lu$ is a minimal graded left ideal. Let
$e\neq f$ be two edges with $s(e)=u=s(f)$. Then $Lee^{\ast}$ is a\ non-zero
proper graded submodule of $Lu$ contradicting the fact that $Lu$ is a minimal
graded left ideal.
\end{proof}

\begin{corollary}
\label{Simple implies at most one edge}If $v$ is a vertex such that $Lv$ is a
minimal graded left ideal, then $|s^{-1}(v)|\leq1$.
\end{corollary}

The following Lemma will be used in the characterization of minimal graded
left ideals of Leavitt path algebras. This is a graded version of a well-known
result on semiprime rings. We will provide a simple proof for the graded
version. Recall that a graded ring $R$ is called graded semiprime if it has no
non-zero nilpotent graded ideals.

\begin{lemma}
\label{eAe graded field implies Le graded simple} Let\ $\Gamma$ be a group,
$R$ a\ semiprime $\Gamma$-graded ring and $\epsilon$ be a homogeneous
idempotent. Then $\epsilon R\epsilon$ is a graded division ring if and only if
$R\epsilon$ (likewise, $\epsilon R$) is a minimal graded left (right) ideal.
If $\epsilon R\epsilon$ is a field, then $R\epsilon$ \ ($\epsilon R$) is a
minimal left (right) ideal.
\end{lemma}

\begin{proof}
Suppose $\epsilon R\epsilon=%
{\displaystyle\bigoplus\limits_{\gamma\in\Gamma}}
\epsilon R_{\gamma}\epsilon$ is a graded\ division ring. Let $0\neq a\in
R\epsilon$ be a homogeneous element. \ It is enough if we show that
$Ra=R\epsilon$. \ Now $aR\neq0$, as $a=a\epsilon\in aR$. \ Since $R$ is
semiprime $(aR)^{2}\neq0$. Hence $aba\neq0$ for some homogeneous element $b$.
\ As $aba=a\epsilon ba\epsilon$, $\epsilon ba\epsilon\neq0$ is a homogeneous
element of the graded division ring $\epsilon R\epsilon$ with identity
$\epsilon$. Consequently, there is a homogeneous element $y$ in $\epsilon
R\epsilon$ such that $y\epsilon ba\epsilon=\epsilon$. Then for any $x\in
R\epsilon$, we have $x=x\epsilon=xy\epsilon ba\in Ra$. Hence $R\epsilon=Ra$,
thus proving $R\epsilon$ is a minimal graded left ideal.

Conversely, suppose $R\epsilon$ is a minimal graded left ideal. Let $\epsilon
x\epsilon$ be a homogeneous element of degree $n$ in $\epsilon R\epsilon$.
Clearly the right multiplication by $\epsilon x\epsilon$ is a non-zero graded
morphism $\theta:R\epsilon(-n)\rightarrow R\epsilon$. Since $R\epsilon$ is a
graded minimal left ideal, $\theta$ is actually a graded isomorphism. Let
$\theta^{-1}(\epsilon)=\epsilon y\epsilon$. Then for any $a\in R\epsilon$ we
have $\theta^{-1}(a)=\theta^{-1}(a\epsilon)=a\theta^{-1}(\epsilon)=a\epsilon
y\epsilon$. Consequently, $\epsilon=\theta^{-1}\theta(\epsilon)=(\epsilon
x\epsilon)(\epsilon y\epsilon)$. Similarly, $\epsilon=\theta\theta
^{-1}(\epsilon)=(\epsilon y\epsilon)(\epsilon x\epsilon)$. \ Hence $\epsilon
x\epsilon$ is invertible and we conclude that $\epsilon R\epsilon$ is a graded
division ring$.$
\end{proof}

\begin{definition}
\label{Defn Laurent vertex}A \ vertex $v$ is called a\textit{\ Laurent vertex}
if $T_{E}(v)$ consists of the set of all vertices on a single path $\gamma=\mu
c$ where $\mu$ is a path without bifurcations starting at $v$ and $c$ is a
\ cycle without exits based on a vertex $u=r(\mu)$.
\end{definition}

\begin{lemma}
\label{corner-line pint-Laurent} For any $v\in E^{0}$, the following hold:

(1) \ If $v$ is a line point, then $vLv\cong_{gr}K$.

(2) If $v$ is a Laurent vertex, then $vLv\cong_{gr}K[x^{n},x^{-n}]$ for some
integer $n\geq1$.
\end{lemma}

\begin{proof}
Now the statement (1) is well-known without the isomorphism being a graded
isomorphism (see e.g. Proposition 2.7, \cite{AMMS}). Here is a short proof to
get a graded version of this statement. We first show that if $p$ and $q$ are
paths such that $v=s(p)=s(q)$, then $q=pt$ or $p=qs$, where $s,t$ are paths in
$E$. Suppose first the $|q|\geq|p|$. If $|p|=0$, then $p=v$ and $q=vq$. We
proceed by induction on $|p|$. Suppose the statement holds when $|p|=n\geq0$.
Let $|p|=n+1$. Write $p=ep^{\prime}$ and $q=fq^{\prime}$ where $e,f\in E^{1}$.
Since $v$ is a line point, $e=f$ and $r(e)$ is also a line point. By
induction, $q^{\prime}=p^{\prime}t$. Then $q=fq^{\prime}=ep^{\prime}t=pt$. The
case when $|p|>|q|$ is similar.

Now consider the non-zero monomial $vpq^{\ast}v$. Note that $s(p)=s(q)=v$. If
$|q|\geq|p|$, then $q=pt$ by the above argument. Thus $vpq^{\ast}v=vpt^{\ast
}p^{\ast}v$. It follows that $s(t)=r(t)$. Since $v$ is a line point, $t$ has
to be a vertex, namely, $r(p)$. Then $vpq^{\ast}v=vpt^{\ast}p^{\ast}v=v$. The
case when $|p|>|q|$ is similar. Thus $vLv\cong_{gr}K$ where $K$ is a graded
ring concentrated in degree $0$.

(2) Since $v$ is a Laurent vertex, $T_{E}(v)$ consists of a single path $\mu
c$ where $\mu$ is a path\ without bifurcations starting at $v$ and ending at
$u$ and, $c$ is a cycle of length $n$ without exits based at $u.$ By Lemma
\ref{No Bifurcation imples iso}, $\ Lv(-n)\cong_{gr}Lu$. So%
\[
vLv\cong_{gr}End(Lv(-n))\cong_{gr}End(Lu)\cong_{gr}uLu
\]
. First observe that, since $c$ is a cycle without exits, $cc^{\ast}%
=u=c^{\ast}c$ and any non-zero term $upq^{\ast}u$, where $p$ and $q$ are
paths, simplifies to an integer power of $c$ or $c^{\ast}$. Consequently, if
$a=%
{\displaystyle\sum\limits_{i}}
uk_{i}p_{i}q_{i}^{\ast}u$ is a non-zero element of $uLu$, then $a$ is a
$K$-linear sum of powers of $c$ and $c^{\ast}$. Also note that in the graded
subring $uLu$, $\deg(c)=n$, the length of $c$ and $\deg(c^{\ast})=-n$. Then
the map sending $u$ to $1$, $c$ to $x^{n}$ and $c^{\ast}$ to $x^{-n}$ defines
a graded isomorphism of $uLu$ to $K[x^{n},x^{-n}]$ whose grading is induced by
that of $K[x,x^{-1}]$.
\end{proof}

\begin{proposition}
\label{Lv simple implies line point or Laurent} Let $v\in E^{0}$. Then $Lv$ is
a minimal graded left ideal if and only if $v$ is either a line point or a
Laurent vertex.
\end{proposition}

\begin{proof}
Suppose $Lv$ is a minimal graded left ideal. Assume $v$ is not a line point.
Then $T_{E}(v)$ contains \ a vertex $u$ which is either a bifurcating vertex
or is the base of a cycle. \ By Lemma
\ref{Simple implies Tree has no bifurcation}, $T_{E}(v)$ contains no
bifurcations and so $u$ is the base of a cycle $c$ which obviously has no
exits. In this case, $T_{E}(v)$ \ consists of a single path $\gamma=\mu c$
where $\mu$ is a path without bifurcations with $v=s(\mu)$ and $u=r(\mu
)=s(c)$. Thus $v$ is a Laurent vertex.

Conversely, suppose $v$ is a line point. Then, by Lemma
\ref{corner-line pint-Laurent}, $vLv\cong_{gr}K$ and so $Lv$ is a minimal left
ideal by Lemma \ref{eAe graded field implies Le graded simple}. $Lv$ is also a
graded left ideal since $v$ is a homogeneous idempotent of degree $0$.
Suppose, on the other hand, $v$ is a Laurent vertex. By Lemma
\ref{corner-line pint-Laurent}, $vLv$ is graded isomorphic to the graded
division ring $K[x^{n},x^{-n}]$, where $n\geq1$. By Lemma
\ref{eAe graded field implies Le graded simple}, $Lv$ is a minimal graded left ideal.
\end{proof}

\begin{proposition}
\label{graded simple iso} A\ graded left ideal $S$ of $L$ is a minimal graded
left ideal of $L$ if and only if $S=La$ where $a$ is a homogeneous element and
$La\cong_{gr}Lu(n)$, \ where $u$ is either a line point or a Laurent vertex
and $n\in%
\mathbb{Z}
$.
\end{proposition}

\begin{proof}
Suppose $\ S=La$ is\ a minimal graded left ideal. We can assume that $a$ is a
homogeneous element. This is because, if $a=a_{1}+\cdot\cdot\cdot+a_{m}$ where
the $a_{i}$ are homogeneous, then each $a_{i}\in S$ as $S$ is graded. Moreover
$La_{i}=S$ by graded-simplicity. So write $S=La$ for some homogeneous element
$a$. By (Proposition 3.1, \cite{AMMS}), there exist paths $\mu,\nu$ such that
$\mu^{\ast}a\nu=kw$ or a polynomial $f(c,c^{\ast})$, where $w$ is a vertex,
$k\in K$, $f(x,x^{-1})\in K[x,x^{-1}]$ and $c$ is a cycle without exits.
Consider the case when $\ \mu^{\ast}a\nu=f(c,c^{\ast})$. Since $\mu^{\ast}%
a\nu$ is homogeneous monomial, $\mu^{\ast}a\nu=k_{i}c^{n_{i}}(c^{\ast}%
)^{m_{i}}$ where $k_{i}\in K$ and $m_{i},n_{i}$ are integers. Then
$\gamma^{\ast}a\delta=k^{\prime}u$ where $\gamma=\mu c^{n_{i}}$, $\delta=\nu
c^{m_{i}}$ and $k^{\prime}\in K$. Thus, in either case, for suitable paths
$\alpha,\beta$ with $\beta$ having no bifurcations, we get that $\alpha^{\ast
}a\beta=ku$ where $u$ is a vertex and $k\in K$. Now the map $\phi
:La(-|\beta|)\rightarrow Lu$ given by $\phi(xa)=xa\beta$ is a non-zero graded
morphism and is indeed a graded isomorphism by the simplicity of $La$. Thus
\ $La(-|\beta|)\cong_{gr}Lu$. By Proposition
\ref{Lv simple implies line point or Laurent}, $u$ is either a line point or a
Laurent vertex.
\end{proof}

From the preceding results, we obtain a description of the graded socle of a
Leavitt path algebra $L$. Recall that for a ring $R$, $Soc^{gr}(R)$ denotes
the sum of all\ minimal graded left ideals of $R$. Similar to the non-graded
case, for semi-prime graded rings $R$ (such as Leavitt path algebras $L$), the
left and the right graded socles coincide. Even though $Soc(L)$ is a graded
ideal, it need not be equal to $Soc^{gr}(L)$, as is clear from considering
$L=K[x,x^{-1}]$. However, as noted in \cite{NvO}, $Soc(R)\subseteq
Soc^{gr}(R)$, for any graded ring $R$. For a Leavitt path algebra $L$, one can
also derive this from Proposition
\ref{Lv simple implies line point or Laurent} and the fact that $Soc(L)$ is
the sum of all the left ideals $Lv$ where $v$ is a line point \cite{AMMS}.

\bigskip

\textbf{Grading of a matrix ring over a }$%
\mathbb{Z}
$\textbf{-graded ring:} In preparation of the structure theorem on graded
socle of $L$ and also for use in Section 6, we wish to recall the grading of
matrices of finite order and then indicate how to extend this to the case of
infinite matrices in which at most finitely many entries are non-zero (see
\cite{H-1} and \cite{NvO}).

Let $\Gamma$ be an additive abelian group, $A$ be a $\Gamma$-graded ring and
$(\delta_{1},\cdot\cdot\cdot,\delta_{n})$ an $n$-tuple where $\delta_{i}%
\in\Gamma$. Then $M_{n}(A)$ is a $\Gamma$-graded ring with the $\lambda
$-homogeneous components of $n\times n$ matrices%

\[
M_{n}(A)(\delta_{1},\cdot\cdot\cdot,\delta_{n})_{\lambda}=\left(
\begin{array}
[c]{cccccc}%
A_{\lambda+\delta_{1}-\delta_{1}} & A_{\lambda+\delta_{2}-\delta_{1}} & \cdot
& \cdot & \cdot & A_{\lambda+\delta_{n}-\delta_{1}}\\
A_{\lambda+\delta_{1}-\delta_{2}} & A_{\lambda+\delta_{2}-\delta_{2}} &  &  &
& A_{\lambda+\delta_{n}-\delta_{2}}\\
&  &  &  &  & \\
&  &  &  &  & \\
&  &  &  &  & \\
A_{\lambda+\delta_{1}-\delta_{n}} & A_{\lambda+\delta_{2}-\delta_{n}} &  &  &
& A_{\lambda+\delta_{n}-\delta_{n}}%
\end{array}
\right)  .\qquad\ \ (1)
\]

This shows that for each homogeneous element $x\in A$,
\[
\deg(e_{ij}(x))=\deg(x)+\delta_{i}-\delta_{j}\text{,}\qquad\qquad\qquad
\qquad(2)
\]
where $e_{ij}(x)$ is a matrix with $x$ in the $ij$-position and with every
other entry $0$.

Now let $A$ be a $\Gamma$-graded ring and let $I$ be an arbitrary infinite
index set. Denote by $M_{I}(A)$ the matrix with entries indexed by $I\times I$
having all except finitely many entries non-zero and for each $(i,j)\in
I\times I$, the $ij$-position is denoted by $e_{ij}(a)$ where $a\in A$.
Considering a "vector" $\bar{\delta}:=(\delta_{i})_{i\in I}$ where $\delta
_{i}\in\Gamma$ and following the usual grading on the matrix ring \ (see
(1),(2)), define, for each homogeneous element $a$,%
\[
\deg(e_{ij}(a))=\deg(a)+\delta_{i}-\delta_{j}\text{.}\qquad\qquad\qquad
\qquad(3)
\]
This makes $M_{I}(A)$ a $\Gamma$-graded ring, which we denote by
$M_{I}(A)(\bar{\delta})$. Clearly, if $I$ is finite with $|I|=n$, then the
graded ring coincides (after a suitable permutation) with $M_{n}(A)(\delta
_{1},\cdot\cdot\cdot,\delta_{n})$.

Suppose $E$ is a finite acyclic graph consisting of exactly one sink $v$. Let
$\{p_{i}:1\leq i\leq n\}$ be the set of all paths ending at $v$. Then it was
shown in (Lemma 3.4, \cite{AAS})
\[
L_{K}(E)\cong M_{n}(K)\qquad\qquad\qquad(4)
\]
under the map $p_{i}p_{j}^{\ast}\longmapsto e_{ij}$. Now taking into account
the grading of $M_{n}(K)$, it was further shown in (Theorem 4.14, \cite{H-1})
that the same map induces a graded isomorphism
\[
L_{K}(E)\longrightarrow M_{n}(K)(|p_{1}|,\cdot\cdot\cdot,|p_{n}|)\qquad
\qquad\qquad(5)
\]%
\[
p_{i}p_{j}^{\ast}\longmapsto e_{ij\text{.}}%
\]
In the case of a comet graph $E$ (that is, a finite graph $E$ in which every
path eventually ends at a vertex on a cycle $c$ without exits), it was shown
in \cite{AAS-1} that the map
\[
L_{K}(E)\longrightarrow M_{n}(K[x,x^{-1}])\qquad\qquad\qquad(6)
\]%
\[
\text{ }p_{i}c^{k}p_{j}^{\ast}\longmapsto e_{ij}(x^{k})
\]
induces an isomorphism. Again taking into account the grading, it was shown in
(Theorem 4.20, \cite{H-1}) that the map%
\[
L_{K}(E)\longrightarrow M_{n}(K[x^{|c|},x^{-|c|})(|p_{1}|,\cdot\cdot\cdot
\cdot,|p_{n}|)\qquad\qquad\qquad(7)
\]%
\[
\text{ }p_{i}c^{k}p_{j}^{\ast}\longmapsto e_{ij}(x^{k|c|}).
\]
induces a graded isomorphism. Later in the paper \cite{AAPS}, the isomorphisms
(4) and (6) were extended to infinite acyclic and infinite comet graphs
respectively \ (see Proposition 3.6 \cite{AAPS}). The same isomorphisms with
the grading adjustments will induce graded isomorphisms for Leavitt path
algebras of such graphs. We now describe this extension below.

Let $E$ be a graph such that no cycle in $E$ has an exit and such that every
infinite path contains a line point or is tail-equivalent to a rational path
$ccc\cdot\cdot\cdot$ ... \ where $c$ is a cycle (without exits). \ Define an
equivalence relation in the set of all line points in $E$ by setting $u\sim v$
if $T_{E}(u)\cap T_{E}(v)\neq\emptyset$\ \ Let \ $X$ be the set of
representatives of distinct equivalence classes of line points in $E$, so that
for any two line points $u,v\in X$ with $u\neq v$, $T_{E}(u)\cap
T_{E}(v)=\emptyset$. For each vertex \ $v_{i}\in X$, let $\overset{-}{p^{v_{i}%
}}:=\{p_{s}^{v_{i}}:s\in\Lambda_{i}\}$ be the set of all paths that end at
$v_{i}$, where $\Lambda_{i}$ is an index set which could possibly be infinite.
Denote by $|\overset{-}{p^{v_{i}}|}=\{|p_{s}^{v_{i}}|:s\in\Lambda_{i}\}$.

Let $Y$ be the set of all distinct cycles in $E$. As before, for each cycle
$c_{j}\in Y$ based at a vertex $w_{j}$, let $\overset{-}{q^{w_{j}}}%
:=(q_{r}^{w_{j}}:r\in\Upsilon_{j}\}$ be the set of all paths that end at
$w_{j}$ that do not include all the edges of $c_{j}$ where $\Upsilon_{j}$ is
an index set which could possibly be infinite. Let $|$ $\overset{-}{q^{w_{j}}%
}|:=\{|q_{r}^{w_{j}}|:r\in\Upsilon_{j}\}$. Then the isomorphisms (5) and (7)
extend to a $%
\mathbb{Z}
$-graded isomorphism%
\[
L_{K}(E)\cong_{gr}%
{\displaystyle\bigoplus\limits_{v_{i}\in X}}
M_{\Lambda_{i}}(K)(|\overset{-}{p^{v_{i}}|)}\oplus%
{\displaystyle\bigoplus\limits_{w_{j}\in Y}}
M_{\Upsilon_{j}}(K[x^{|c_{j}|},x^{-|c_{j}|}])(|\overset{-}{q^{w_{j}}}%
|)\qquad(8)
\]
where the grading is as in (3).

We are now ready to prove the main theorem of this section.

\begin{theorem}
\label{Graded Socle}Let $E$ be an arbitrary graph and $L=L_{K}(E)$. Then the
graded socle $Soc^{gr}(L)$ is the two-sided ideal generated by the set
consisting of all the line points and all the Laurent vertices in $E$ and
\[
Soc(L)^{gr}\cong_{gr}%
{\displaystyle\bigoplus\limits_{v_{i}\in X}}
M_{\Lambda_{i}}(K)(|\overset{-}{p^{v_{i}}|)}\oplus%
{\displaystyle\bigoplus\limits_{w_{j}\in Y}}
M_{\Upsilon_{j}}(K[x^{t_{j}},x^{-t_{j}}])(|\overset{-}{q^{w_{j}}}|)
\]
where $\Lambda_{i},\Upsilon_{j}$ are suitable index sets, the $t_{j}$ are
positive integers, \ $X$ is the set of representatives of distinct equivalence
classes of line points in $E$ and $Y$ is the set of all distinct cycles
(without exits) in $E$.
\end{theorem}

\begin{proof}
Let $I$ be the two-sided ideal of $L$ generated by all the line points and all
the Laurent vertices in $E$. Now $I\subseteq Soc^{gr}(L)$ since, by
Proposition \ref{Lv simple implies line point or Laurent}, $Lu\subset
Soc^{gr}(L)$ if $u$ is a line point or a Laurent vertex. To prove the reverse
inclusion, suppose $S=La$ is a minimal graded left ideal of $L$, where $a$ is
a homogeneous element. It is clear from Proposition \ref{graded simple iso}
and its proof,\ that there is a vertex $v$ which is either a line point or a
Laurent vertex and a path $\beta$ without exits such that $La(-|\beta
|)=Lv\beta^{\ast}$. Consequently, $La(-|\beta|)\subset I$. Since $I$ is a
graded ideal, $S=La\subset I$. This proves that $I=Soc^{gr}(L)$. \ Let $A$ be
the two-sided ideal generated by the set $H$ of all the line points in $E$ and
$B$ be the two-sided ideal generated by the set $H^{\prime}$ of all the
Laurent vertices in $E$. Clearly, $H\cap H^{\prime}=\varnothing$. This implies
that the intersection $\bar{H}\cap\bar{H}^{\prime}$ of the saturated closures
of $H,H^{\prime}$ is empty (See Lemma 1.4, \cite{ARS}). Hence $A\cap B=0$. So
$Soc^{gr}(L)=A\oplus B$. By Theorem 3.7 and Propositions 3.5 and 3.6 of
\cite{AAPS}, there are algebraic isomorphisms $A\cong%
{\displaystyle\bigoplus\limits_{i\in I}}
M_{\Lambda_{i}}(K)$\ and $B\cong%
{\displaystyle\bigoplus\limits_{j\in J}}
M_{\Upsilon_{j}}(K[x^{t_{j}},x^{-t_{j}}])$ where $\Lambda_{i},\Upsilon_{j}$
are arbitrary index sets and $t_{j}$ are integers $\geq1$. Now $A$ and $B$ are
graded ideals of $L$ and the indicated isomorphisms induce the graded
isomorphism stated in (8) above, thus showing that
\[
Soc(L)^{gr}\cong_{gr}%
{\displaystyle\bigoplus\limits_{v_{i}\in X}}
M_{\Lambda_{i}}(K)(|\overset{-}{p^{v_{i}}|)}\oplus%
{\displaystyle\bigoplus\limits_{w_{j}\in Y}}
M_{\Upsilon_{j}}(K[x^{t_{j}},x^{-t_{j}}])(|\overset{-}{q^{w_{j}}}|)\text{.}%
\]

\end{proof}

\textbf{Remark}: In (Theorem 3.9, \cite{AAPS}), it is proved that a Leavitt
path algebra $L$ is categorically noetherian (that is, submodules of finitely
generated left/right $L$-modules are finitely generated) exactly when $L$
$\cong%
{\displaystyle\bigoplus\limits_{i\in I}}
M_{\Lambda_{i}}(K)\oplus%
{\displaystyle\bigoplus\limits_{j\in J}}
M_{\Lambda_{j}}(K[x,x])$ where $\Lambda_{i},\Lambda_{j}$ are suitable index
sets and the $t_{j}$ are positive integers. In view of Theorem
\ref{Graded Socle}, we conclude that a Leavitt path algebra $L$ over an
arbitrary graph is categorically noetherian if and only if $L$ coincides with
its graded socle.

\section{Graded E-Algebraic Branching Systems and Graded Simple Modules}

Given an arbitrary graph $E$, the concept of an\textit{ E-algebraic branching
system }was studied by Gon\c{c}alves and Royer in \cite{GR} and was used to
provide representations for a Leavitt path algebra $L$ of the graph $E$. \ In
this section, we define a grading on these algebraic branching systems \ and
use the graded systems to construct various graded-simple modules over $L$. We
also describe the corresponding annihilating ideals of these graded-simple
$L$-modules. It may be of some interest to note that, unlike the ungraded
case, the annihilator ideal of a graded-simple $L$-module need not be
primitive, but it is always a graded prime ideal of $L$. \ Conversely, when
$E$ is row-finite or $E^{0}$ is countable, we show that every graded prime
ideal of $L$ which is not primitive occurs as the annihilator of some
graded-simple $L$-module which is not simple. Using infinite paths and sinks
in $E$, Chen \cite{C} constructed simple $L$-modules and showed that these
simple modules can also be constructed by using algebraic branching systems.
This approach was further expanded in \cite{ARa} and \cite{R-2} to introduce
new types of simple $L$-modules induced by vertices which are infinite
emitters in $E$. All these new modules are now called Chen simple $L$-modules.
Examples show that not all graded-simple $L$-modules are Chen simple modules
and likewise, not all Chen simple $L$-modules are graded-simple. We show that
the Chen simple module $V_{[p]}$ corresponding to an infinite path $p$ is
graded-simple if and only if $p$ is an infinite irrational path in $E$.

\begin{definition}
Let $E$ be an arbitrary graph. An $E$\textit{-algebraic branching system
}consists of a set $X$ and a family of its subsets $\{X_{v},X_{e}:v\in
E^{0},e\in E^{1}\}$ such that

(1) $X_{v}\cap X_{w}=\emptyset=X_{e}\cap X_{f}$ for $v,w\in E^{0}$ with $v\neq
w$ and $e,f\in E^{1}$ with $e\neq f$;

(2) $X_{e}\subseteq X_{s(e)}$ for $e\in E^{1}$;

(3) For all $v\in E^{0}$, $X_{v}=%
{\displaystyle\bigcup\limits_{e\in s^{-1}(v)}}
X_{e}$;

(4) For each $e\in E^{1}$, there exists a bijection $\sigma_{e}:X_{r(e)}%
\rightarrow X_{e}$.
\end{definition}

For our purposes, we also assume that $X$ is saturated, that is, $X=%
{\displaystyle\bigcup\limits_{v\in E^{0}}}
X_{v}$.

\begin{definition}
The $E$-branching system is called \textit{graded} if there is a map
$\deg:X\rightarrow%
\mathbb{Z}
$ such that

(5) $\deg(\sigma_{e}(x))=\deg(x)+1$.
\end{definition}

Given an $E$-algebraic branching system $X$, we define a left module over
$L:=L_{K}(E)$ as follows: Consider the $K$-vector space $M(X)$ having $X$ as a
basis. Following \cite{GR}, we define, for each vertex $v$ and each edge $e$
in $E$, linear transformations $P_{v},S_{e}$ and $S_{e^{\ast}}$ on $M(X)$ as follows:

\begin{definition}
For all paths $x\in X$,

(I) $\ \ P_{v}(x)=\left\{
\begin{array}
[c]{c}%
x\text{, if }x\in X_{v}\\
0\text{, otherwise}%
\end{array}
\right.  $

(II) $\ S_{e}(x)=\left\{
\begin{array}
[c]{c}%
\sigma_{e}(x)\text{, if }x\in X_{r(e)}\\
0\text{, \qquad otherwise}%
\end{array}
\right.  $

(III) $S_{e^{\ast}}(x)=\left\{
\begin{array}
[c]{c}%
\sigma_{e}^{-1}(x)\text{, if }x\in X_{e}\\
0\text{, \qquad otherwise}%
\end{array}
\right.  $
\end{definition}

Then it is straightforward to check that the endomorphisms $\{P_{u}%
,S_{e},S_{e^{\ast}}:u\in E^{0},e\in E^{1}\}$ satisfy the defining relations
(1) - (4) of the Leavitt path algebra $L$ (see \cite{GR}). This induces an
algebra homomorphism $\phi$ from $L$ to $End_{K}(M(X))$ mapping $u$ to $P_{u}%
$, $e$ to $S_{e}$ and $e^{\ast}$ to $S_{e^{\ast}}$ (\cite{GR}, Theorem 2.2).
Then $M(X)$ can be made a left module over $L$ via the homomorphism $\phi$. We
denote this $L$-module operation on $M(X)$ by $\cdot$.

If $X$ is a graded branching system, then define, for each $i\in%
\mathbb{Z}
$, the homogeneous component
\[
M(X)_{i}=\{%
{\displaystyle\sum\limits_{x\in X}}
k_{x}x\in M(X):\deg(x)=i\}\text{.}%
\]
It is easy to see that
\[
M(X)=%
{\displaystyle\bigoplus\limits_{i\in\mathbb{Z}}}
M(X)_{i}%
\]
and that $M(X)$ is a $%
\mathbb{Z}
$-graded left $L$-module.

We shall now provide three illustrations of constructing graded-simple
$L$-modules using appropriate $E$-algebraic branching systems.

\bigskip

A) \textbf{A graded-simple module which is not simple:} Let $E$ be a graph.
Let $u\in E^{0}$ be a Laurent vertex so that $T_{E}(u)$ consists of a single
path $\gamma=\mu c$ where $\gamma$ is a path without bifurcations and $c$ is a
cycle without exits based at a vertex $v$. Let $X=\{pq^{\ast}:p,q$ paths with
$r(q^{\ast})=v\}$. We make $X$ a graded branching system as follows:

$X_{w}=\{pq^{\ast}\in X:s(p)=w\}$, where $w\in E^{0};$

$X_{e}=\{pq^{\ast}\in X:e$ initial edge of $p\}$, where $e\in E^{1}$;

$\sigma_{e}:X_{r(e)}\rightarrow X_{e}$ given by $pq^{\ast}\longmapsto
epq^{\ast}$, a bijection;

$\deg:X\rightarrow%
\mathbb{Z}
$, given by $pq^{\ast}\longmapsto|pq^{\ast}|=|p|-|q|$.

The only non-trivial condition of the algebraic branching that needs to be
\ checked here is $X_{w}=%
{\displaystyle\bigcup\limits_{e\in s^{-1}(w)}}
X_{e}$. Let $pq^{\ast}\in X_{w}$. If $|p|\geq1$ then clearly $p=et$, where
$s(e)=w$. Thus $pq^{\ast}=etq^{\ast}\in X_{e}$. If $|p|=0$, then $pq^{\ast
}=q^{\ast}$ such that $s(q^{\ast})=r(q)=w$ and, by hypothesis on $X$,
$r(q^{\ast})=s(q)=v$. Since $\gamma=\mu c$ has no exits, the path $q$ is part
of $\gamma$ and $w$ will emit only one edge $e$. Then $q^{\ast}=ee^{\ast
}q^{\ast}\in X_{e}$ thus showing\ $X_{w}\subseteq%
{\displaystyle\bigcup\limits_{e\in s^{-1}(w)}}
X_{e}$. The converse inclusion is obvious.

Then conditions (I),(II),(III) make $M(X)$ a graded $L$-module which we denote
by $N_{vc}$.

Next we wish to list some of the properties of module $N_{vc}$ in Theorem
\ref{Simple module N-vc}. Its proof needs the following descriptions of the
graded prime and graded primitive ideals of a Leavitt path algebra $L$ of an
arbitrary graph $E$ stated in Theorems 3.12 and 4.3 of \cite{R-1}. Recall
(\cite{ABR}) , subset $X$ of $E^{0}$ is said satisfy the \textit{countable
separation property} if there is a countable subset $Y$ of $E^{0}$ with the
property that to each $u\in X$, there is a $v\in Y$ such that $u\geq v$.

\begin{theorem}
\label{Rangas thm on primitive ideals}(\cite{R-1}): Let $P$ be an ideal of $L$
with $P\cap E^{0}=H$. Then

(a) $P$ is a graded prime ideal if and only if either (i) $P=I(H,B_{H})$
(where $B_{H}$ may be empty) with $E^{0}\backslash H$ downward directed or
(ii) $P=I(H,B_{H}\backslash\{u\})$ for some $u\in B_{H}$ and $w\geq u$ for
every $w\in E^{0}\backslash H$.

(b) $P$ is a graded primitive ideal if and only if either (i) $P$ is a graded
prime ideal of the form $I(H,B_{H})$ such that $E\backslash H$ satisfies
Condition (L) and the countable separation property or (ii) $P$ is the prime
ideal of type (ii) described above.
\end{theorem}

The next theorem lists some properties of the module $N_{vc}$.

\begin{theorem}
\label{Simple module N-vc}Let $v$ be a Laurent vertex based at a cycle $c$.
Then we have the following:

(1) The module $N_{vc}$ is a graded-simple $L$-module but is not a simple $L$-module;

(2).The annihilator ideal of $N_{vc}$ is the ideal $I(H(v),\emptyset)$ where
$H(v)=\{u\in E^{0}:u\ngeqq v\}$. It is a graded prime ideal of $L$ which is
not primitive;

(3) Conversely, if $E$ is row-finite or if $E^{0}$ is countable, then every
graded prime ideal $P$ of $L$ which is not primitive is the annihilator of a
graded-simple module which is not simple.
\end{theorem}

\begin{proof}
(1) Suppose $U$ is a non-zero graded submodule of $N_{vc}$. Let $a$ be a
non-zero homogeneous element in $U$, say, $a=%
{\displaystyle\sum\limits_{i=1}^{n}}
k_{i}p_{i}q_{i}^{\ast}$, where, by definition, $r(q_{i}^{\ast})=v$.
Multiplying by $p_{1}^{\ast}$ on the left, and replacing $a$ by $p_{1}^{\ast
}a$, we may assume that the first term of $a$ is $k_{1}q_{1}^{\ast}$. Again
multiplying on the left by $s(q_{1}^{\ast})$ and gathering all the non-zero
terms, we may assume without loss of generality that
\[
a=k_{1}q_{1}^{\ast}+k_{2}p_{2}q_{2}^{\ast}+\cdot\cdot\cdot+k_{r}p_{r}%
q_{r}^{\ast}%
\]
where $|q_{1}|>0,s(p_{i})=s(q_{1}^{\ast})=r(q_{1})$ and $s(q_{1}%
)=r(q_{1}^{\ast})=r(q_{i}^{\ast})=s(q_{i})=v$ for all $i$. Since the path $c$
has no exits, $q_{1}$ lies entirely on $c$ and this implies that the $p_{i}$
lie entirely on $c$. Further the $q_{i}$ also lie on $c$ with $s(q_{i})=v$ and
$r(q_{i})=r(p_{i})$. Since $\deg(p_{i}q_{i}^{\ast})=\deg(q_{1}^{\ast})$ for
all $i$, we conclude that each $p_{i}q_{i}^{\ast}=q_{1}^{\ast}$. So
$a=kq_{1}^{\ast}$ for some $k\in K$. Then $k^{-1}q_{1}a=v\in U$ and so
$U=N_{v\infty}$. Thus $N_{v\infty}$ is a graded-simple $L$-module. But
$N_{vc}$ is not a simple $L$-module. Because, if $f$ is an edge with $r(f)=v$,
then it can be readily seen that $v\notin L(v+f)$, thus showing that $(v+f)$
generates a non-zero proper submodule of $N_{vc}$.

(2) First observe that the annihilator $Q$ of a graded simple module $S$ is
always a graded prime ideal: $Q$ is graded, because if $a\in Q$ with
$a=a_{1}+\cdot\cdot\cdot+a_{n}$ a graded sum of homogeneous elements $a_{i}$,
then for any homogeneous element $x\in S$, $0=ax=a_{1}x+\cdot\cdot\cdot
+a_{n}x$ implies $a_{i}x=0$ for all $i$. Hence each $a_{i}\in Q$. To see that
$Q$ is prime, suppose $A,B$ are graded ideals such that $A\nsubseteq Q$ and
$B\nsubseteq Q$. Then $AS$ is a non-zero graded submodule of $S$ and so
$AS=S$, by simplicity. Similarly $BS=S$. Then $ABS=AS=S$. Hence $AB\nsubseteq
Q$, thus proving that $Q$ is graded prime.. Thus the annihilator $P$ of
$N_{vc}$ must be a graded prime ideal of $L$. From the description of $N_{vc}$
it is clear that $H(v)=\{u\in E^{0}:u\ngeqq v\}\subset P$. Actually, $P\cap
E^{0}=H(v)$, since $w\geq v$ for all $w\in E^{0}\backslash H(v)$. Also, since
$B_{H}\subseteq E^{0}\backslash H$, we have for each $u\in B_{H}$ a path $p$
connecting $u$ to $v$. Clearly then
\[
u^{H(v)}p=(u-%
{\displaystyle\sum\limits_{e\in s^{-1}(u),r(e)\in H(v)}}
ee^{\ast})p=p\neq0.
\]
So $u^{H(v)}\notin P$. It then follows from Theorem
\ref{Rangas thm on primitive ideals}(a) that $P=I(H(v),\emptyset)$. As
$E\backslash H(v)$ contains the cycle $c$ without exits, it does not satisfy
Condition (L) and hence by Theorem \ref{Rangas thm on primitive ideals}(b)
that $P$ is not a primitive ideal of $L$.

(3) \ Suppose $P$ is a graded prime ideal of $L$ which is not primitive. Let
$H=P\cap E^{0}$. \ By Theorem \ref{Rangas thm on primitive ideals},
$E^{0}\backslash H$ is downward directed and $P=I(H,B_{H})$. Also,
$E\backslash H$ satisfies the countable separation property whenever
$E\backslash H$ is row-finite or $E^{0}\backslash H$ is countably infinite
\cite{ABR}. Then the hypothesis that $P$ is not primitive implies, by Theorem
\ref{Rangas thm on primitive ideals}, that $E\backslash H$ does not satisfy
Condition (L). Hence there is a cycle $c$ based at a vertex $v$ and without
exits in $E\backslash(H,B_{H})$. Thus $v$ is a Laurent vertex in $E\backslash
H$. The denoting $\bar{v}=v+P$, we appeal to Proposition
\ref{Lv simple implies line point or Laurent} to conclude that \ $S=(L/P)\bar
{v}\cong L_{K}(E\backslash(H,B_{H}))v$ is a graded-simple $L/P$-module which
obviously not a simple $L/P$-module (as $v$ is not a line point in
$E\backslash(H,B_{H})$). Since $w\geq v$ for every vertex $w\in E\backslash
(H,B_{H})$, \ it is then clear that the annihilator of $S$ as an $L/P$-module
is $\{0\}$. Consequently, as an $L$-module $S$ is a graded-simple module which
is not simple and its annihilator ideal is $P$.
\end{proof}

\bigskip

B) \textbf{A graded-simple module which is also simple: }Let $E$ be any graph
\ and let $v\in E^{0}$ be an infinite emitter. Let $X$ be the set of all the
finite paths ending at $v$. We make $X$ a graded branching system as follows:

$X_{w}=\{q\in X:s(q)=w\}$, where $w\in E^{0}$;

$X_{e}=\{q\in X:e$ is the initial edge of $q\}$, where $e\in E^{1}$;

A bijection $\sigma_{e}:X_{r(e)}\rightarrow X_{e}$ given by $\sigma_{e}(q)=eq$;

$\deg:X\rightarrow%
\mathbb{Z}
$ is given by $\deg(p)=|p|$.

Then the $K$-vector space $M(X)$\ having $X$ as a basis becomes a left
$L$-module by using the condition (I),(II),(III). It becomes a graded module
by defining, for each $i\in%
\mathbb{Z}
$,%
\[
M(X)_{i}=\{%
{\displaystyle\sum\limits_{p_{i}\in X}}
k_{i}p_{i}\in M(X):|p_{i}|=i\}
\]
and $M(X)$ $=%
{\displaystyle\bigoplus\limits_{i\in\mathbb{Z}}}
M(X)_{i}$. \ Now the module $M(X)$ is the same thing as the module
$S_{v\infty}$ introduced in \cite{R-2} , since the construction $M(X)$ is
essentially a reformulation, using the E-algebraic branching systems, of the
construction of the module $S_{v\infty}$. In (Proposition 2.3, \cite{R-2}), it
was shown\ that $S_{v\infty}$ becomes a simple $L$-module. Since $S_{v\infty
}=M(X)$ is also graded, $S_{v\infty}$ is a graded-simple left $L$-module. As
noted in \cite{R-2}, $S_{v\infty}$ gives rise to three distinct classes of
non-isomorphic simple and hence graded-simple $L$-modules depending on the
vertex $v$ satisfying one of the following three properties, where
$N(v)=\{u\in E^{)}:u\geq v\}$:

(i) $\ \ |s^{-1}(v)\cap r^{-1}(N(v))|=0;$

(ii) \ $0<\ |s^{-1}(v)\cap r^{-1}(N(v))|<\infty;$

(iii) $\ |s^{-1}(v)\cap r^{-1}(N(v))|$ is infinite.

\textbf{Remark}: (i) Suppose $w$ is a sink in $E$. Define $X$ to be the set of
all finite paths ending at $w$. Proceed as in paragraph B) above to make $X$ a
graded algebraic branching system and construct the graded $L$-module $M(X)$.
This module was denoted by $N_{w}$ in (\cite{C}) where it was shown to be
simple. It is thus a graded-simple $L$-module which is also simple.

(ii) The above approach to defining a graded simple module does not work if
$v$ is a regular vertex. This is because, in this case, the linear
transformations $P_{v},S_{e},S_{e^{\ast}}$ do not satisfy the CK-2 relation
\ $%
{\displaystyle\sum\limits_{e\in s^{-1}(v)}}
S_{e}S_{e^{\ast}}=P_{v}$. Because on the one hand $(%
{\displaystyle\sum\limits_{e\in s^{-1}(v)}}
S_{e}S_{e^{\ast}})(v)=0$ but on the other hand $P_{v}(v)=v\neq0$..

\bigskip

C) \textbf{A simple module which is not graded-simple: }Let $E$ be any graph
and let $p$ be an infinite path in $E$. Define \ an $E$-algebraic branching
system with $X=[p]$, the set of all infinite paths tail-equivalent to $p$.

For each $v\in E^{0}$, let $X_{v}=\{q\in\lbrack p]:s(q)=v\}$

For each $e\in E^{1}$, let $X_{e}=\{q\in\lbrack p]:e$ is the initial edge of
$q\}.$

Define $\sigma_{e}:X_{r(e)}\rightarrow X_{e}$ by $q\longmapsto eq$.

It was shown by Chen in (\cite{C}) that $M(X)$ becomes a simple left
$L$-module and denoted it by $V_{[p]}.$

Next we show that \ infinite rational paths give rise to non-gradable simple
$L$-modules. We begin with a simple lemma.

\begin{proposition}
\label{V_[p] graded when p irrational}Let $E$ be an arbitrary graph and let
$p$ be an infinite path in $E$. Then the corresponding Chen simple left
$L$-module $V_{[p]}$ is a graded-simple $L$-module if and only if $p$ is an
irrational infinite path in which case $p$ is a homogeneous element of
$V_{[p]}$. .
\end{proposition}

\begin{proof}
Assume $p$ is a rational path, say $p=ccc\cdot\cdot\cdot$ where $c$ is a
cycle. Suppose, by way of contradiction, that $V_{[p]}$ becomes a graded
module under the defined $L_{K}(E)$-module operation. First we claim that $p$
must be a homogeneous element. Suppose, on the contrary,
\[
p=%
{\displaystyle\sum\limits_{i=1}^{n}}
h_{i}\qquad\qquad\qquad\qquad(++)
\]
where the $h_{i}$ are non-zero homogeneous elements in $V_{[p]}$ and
$h_{i}\neq p$ for all $i$. Since $h_{i}\in V_{[p]}$, write $h_{i}=%
{\displaystyle\sum\limits_{j=1}^{t_{i}}}
k_{ij}q_{ij}$ where $k_{ij}\in K$, $q_{ij}\in\lbrack p]$ and the paths
$q_{ij}$ are distinct for all $i,j$. Choose an integer $n$ such that the paths
$p^{\leq n}$ and the $q_{ij}^{\leq n}$ are distinct for all $i,j$. Multiplying
the equation $(++)$ on the left with $(c^{\ast})^{-n}$, we then get $p=0$, a
contradiction. Thus $p$ is homogeneous. Since $p=cp,$we then get%
\[
\deg(p)=\deg(cp)=\deg(p)+|c|,
\]
which is a contradiction. Hence $V_{[p]}$ cannot be graded as an $L$-module.

Conversely, suppose $p$ is an infinite irrational path in $E$. For each
$q\in\lbrack p]$, define $\deg(q)=m-n$, if $m$ is the smallest positive
integer such that $\tau^{>m}(q)=\tau^{>n}(p)$ for some $n$. It is then easy to
check that the simple left $L$-module $V_{[p]}$ becomes a graded $L$-module.
The proof that the irrational path $p$ is homogeneous in $V_{[p]}$ is the same
as was given in the preceding paragraph where, instead of multiplying the
equation $(\#)$ on the left by $(c^{\ast})^{-n}$, multiply by $(p^{\leq
n})^{\ast}$ to get a contradiction.
\end{proof}

\section{Leavitt Path Algebras over which every graded-simple module is
finitely presented}

Leavitt path algebras $L_{K}(E)$ over which every non-graded simple modules
are finitely presented seem to have interesting properties. These were
completely described in \cite{ARa} when $E$ is a finite graph and in
\cite{R-3} when $E$ is an arbitrary graph. In this section, we provide a
complete description of the Leavitt path algebra $L$ of an arbitrary graph $E$
over which every graded-simple is finitely presented.

First observe that if every simple $L$-module is finitely presented, then
every graded-simple module need not be finitely presented. \ For example, let
$L$ be the algebraic Toeplitz algebra so that $L=L_{K}(E)$, with $E$
consisting of two vertices $v,w$ and two edges $e,f$ such that
$v=s(e)=r(e)=s(f)$ and $w=r(f)$. It was shown in (Example 4.5, \cite{R-2})
that every simple $L$- module is finitely presented. This also follows from
Theorem 1.1 of \cite{ARa}. On the other hand, if $I$ denotes the ideal
generated by $w$, then $I$ is a graded ideal and $L/I\cong K[x,x^{-1}]$ is a
graded-simple $L$-module. But $L/I$ is not finitely presented since
$I=Lw\oplus%
{\displaystyle\bigoplus\limits_{n=0}^{\infty}}
Lf^{\ast}(e^{\ast})^{n}$ is infinitely generated. \ As we shall see later, the
condition that all the graded-simple modules are finitely presented \ places
considerable restriction on the nature of the Leavitt path algebra $L$.

In \cite{AMT}, the authors investigate when the Chen simple module $V_{[p]}$
corresponding to an infinite path $p$ is finitely presented. One of their key
results (see Theorem 2.20, \cite{AMT}) states, in essence, that if $p$ is an
infinite irrational path no vertex on which is either an infinite emitter or a
line point, then the corresponding Chen simple $L$-module $V_{[p]}$ is not
finitely presented. But the proof of this important result involves lengthy
computations and it depends on the additional requirement that the vertices on
this path $p$ must be finite emitters. \ Using Theorem \ref{Theorem of Hazrat}
below, we shall show how to get rid of this additional requirement for the
vertices on the path $p$ and to provide a short proof of a more general result
(Theorem \ref{V_p finitely presented}).

We first consider a graded reformulation of Proposition 2 in \cite{AR-1}
\ which was proved without any reference to the grading of the Leavitt path
algebra and which is very useful in this and the next section. A careful
inspection of the construction of the morphism in their proof shows that the
stated homomorphism $\theta$, mentioned in Theorem \ref{Gene-Rnga Theorem}
below, is indeed a graded morphism whose image is a graded submodule of $L$.
We reformulate this theorem reflecting the grading of $L$. \ 

\begin{theorem}
\label{Gene-Rnga Theorem}(\cite{AR-1}) Let $E$ be an arbitrary graph. Then the
Leavitt path algebra $L:=L_{K}(E)$ is a directed union of graded subalgebras
$B=A\oplus K\epsilon_{1}\oplus\cdot\cdot\cdot\oplus K\epsilon_{n}$ where $A$
is the image of a graded homomorphism $\theta$ from a Leavitt path algebra
$L_{K}(F_{B})$ with $F_{B}$ a finite graph (depending on $B$), the
$\epsilon_{i}$ are homogeneous orthogonal idempotents and $\oplus$ is a ring
direct sum. Moreover, if $E$ is acyclic, so is each graph $F_{B}$ and in this
case $\theta$ is a graded monomorphism.
\end{theorem}

\begin{proof}
For the proof of the first statement, see \cite{AR-1}. The second statement
that $\theta$ is a monomorphism follows from the fact that, when $E$ is
acyclic, the mentioned finite graph $F_{B}$ (being acyclic) satisfies
Condition (L) and since $\theta$ does not vanish on the vertices of $F_{B}$,
$\theta$ is a monomorphism by the graded uniqueness theorem.
\end{proof}

Another useful result is the following theorem proved in \cite{H-2}. This
theorem was first proved for finite graphs $E$. Then it was extended to
row-finite graphs and then to arbitrary graphs by the so-called
desingularization process. However the desingularization of a graph does not
always preserve the grading of the algebra. We show below how to fix this
problem by using Theorem \ref{Gene-Rnga Theorem} and deriving a shorter proof
of the general case when $E$ is an arbitrary graph. Recall that a graded ring
$R$ is called \textit{graded von Neumann regular} if for each homogeneous
element $x$ there is a homogeneous element $y$ such that $xyx=x$.

\begin{theorem}
\label{Theorem of Hazrat}(\cite{H-2}) Every Leavitt path algebra $L$ of an
arbitrary graph $E$ is a graded von Neumann regular ring.
\end{theorem}

\begin{proof}
It was shown in (Theorem 10 (I), \cite{H-2}) that if $E$ is any finite graph,
then $L$ is graded von Neumann regular. Suppose now $E$ is an arbitrary graph.
\ By Theorem \ref{Gene-Rnga Theorem}, $L$ is a directed union of graded
subalgebras $B=A\oplus K\epsilon_{1}\oplus\cdot\cdot\cdot\oplus K\epsilon_{n}$
where $A$ is the image of a graded homomorphism $\theta$ from a Leavitt path
algebra $L_{K}(F_{B})$ with $F_{B}$ a finite graph (depending on $B$), the
$\epsilon_{i}$ are orthogonal idempotents and $\oplus$ is a ring direct sum.
Since $F_{B}$ is a finite graph, $L_{K}(F_{B})$ and hence $B$ is graded von
Neumann regular by \cite{H-2}. It is then clear from the definition that the
directed union $L$ is also graded von Neumann regular.
\end{proof}

We summarize some of the properties of graded modules over graded rings
without identity (such as the Leavitt path algebras) that we will be using.
Even though a ring $R$ without identity may not be projective as an
$R$-module, it was shown in \cite{ARS} that a Leavitt path algebra $L$ is
always projective as a left (right) $L$-module. Since $L$ is also graded, it
becomes a graded-projective $L$-module. \ Now $L$ is graded von Neumann
regular by Theorem \ref{Theorem of Hazrat}. We shall be using the following
graded version of a known result which can be obtained by a simple induction
on the number of generators:
\[
\text{\textit{A finitely generated graded left ideal of }}L\text{\textit{ is a
direct summand of }}L\text{\qquad(**)\textit{ }}\mathit{\ }%
\]

We also wish to note that, in spite of the fact that $L$ may not have a
multiplicative identity, maximal graded left ideals always exist in $L$.
Because, if $v\in E^{0}$, then by Zorn's Lemma, $Lv$ contains a maximal graded
$L$-submodule $N$. Then $N\oplus%
{\displaystyle\bigoplus\limits_{u\in E^{0},u\neq v}}
Lu$ is a maximal graded left ideal of $L$. Indeed, every maximal graded left
ideal $M$ of $L$ is of this form: Given $M$ there is a unique vertex $v$ such
that $M=N\oplus%
{\displaystyle\bigoplus\limits_{u\in E^{0},u\neq v}}
Lu$ where $N=Mv=M\cap Lv$ is a maximal graded $L$-submodule of $Lv$ (see Lemma
3.1, Remark 3.2, \cite{R-2}). Consequently every graded-simple $L$-module $S$
of $L$ is graded isomorphic to $Lv/N$ \ for some vertex $v$ and some maximal
graded submodule $N$ of $Lv$.

The following Lemma is well-known for non-graded modules over rings with
identity. We will modify the proof since $L$ need not have an identity and
since we are dealing with graded modules.

\begin{lemma}
\label{All simple projective}If every graded-simple left $L$-module is
projective, then $L$ is graded semi-simple, that is, $L$ is a direct sum of
minimal graded left ideals.
\end{lemma}

\begin{proof}
Our hypothesis implies that every maximal graded left ideal of $L$ is a direct
summand and so $L$ contains minimal graded left ideals. By Zorn's Lemma choose
a maximal family $%
\mathcal{F}%
$ of independent minimal graded left ideals of $L$ and let $M=%
{\displaystyle\bigoplus\limits_{X\in\mathcal{F}}}
X$. We claim that $M=L$. Because, otherwise, there is a vertex $v\notin M$. In
this case, we can embed (by Zorn's Lemma) $M\cap Lv$ inside a maximal graded
$L$-submodule $N$ of $Lv$ which is easily seen as a direct summand of $Lv$,
say $Lv=N\oplus X^{\prime}$. Then $%
\mathcal{F}%
\cup\{X^{\prime}\}$ contradicts the maximality of $%
\mathcal{F}%
$. Thus $L=%
{\displaystyle\bigoplus\limits_{X\in\mathcal{F}}}
X$ is a direct sum of minimal graded left ideals.
\end{proof}

Before we state our next theorem, we consider the following preparatory lemma.

\begin{lemma}
\label{Infinite emitter not fp}Let $v\in E^{0}$ be an infinite emitter. Then
the corresponding graded-simple module $S_{v\infty}$ (considered in section 3)
is not finitely presented.
\end{lemma}

\begin{proof}
Suppose, on the contrary, $S_{v\infty}$ is finitely presented. Note that
$S_{v\infty}\cong Lv/N$ for some $N$\ which will be finitely generated by
supposition. Then, by statement (**) above, $S_{v\infty}$ will be a
graded-projective $L$-module. This means that $S_{v\infty}$ is isomorphic to a
minimal graded left ideal of $L$ and hence, by Proposition
\ref{graded simple iso}, $S_{v\infty}\cong Lu(n)$ where $u$ is either a line
point or a Laurent vertex. Both cases lead to a contradiction. Because if $u$
is a line point, then by Chen \cite{C}, $Lu\cong N_{w}$ or $V_{[p]}$ according
as $T_{E}(u)$ is a finite path ending at a sink $w$ or an infinite path $p$
and this is in contradiction to Theorem 2.9 of \cite{R-2} \ whose proof shows
that $S_{v\infty}$ is not isomorphic to either $N_{w}$ or to $V_{[p]}$.
Likewise, if $u$ is a Laurent vertex, then $Lu$, by Theorem
\ref{Simple module N-vc} (1), cannot be a simple left ideal, but it was shown
in \cite{R-2} that $S_{v\infty}$ is indeed a simple left $L$-module. These
contradictions show that $S_{v\infty}$ cannot be finitely presented.
\end{proof}

\begin{theorem}
\label{All graded-simple fp} Let $E$ be an arbitrary graph . Then the
following are equivalent for $L:=L_{K}(E)$.

(1) Every graded-simple left $L$-module is finitely presented;

(2) $L$ is graded semi-simple, that is, $L$ is a direct sum of minimal graded
left ideals;

(3) The graph $E$ is row-finite and for each vertex $u$ there is an integer
$n\geq0$ such that every path beginning with $u$ and having length $\geq n$
ends at a line point or a Laurent vertex.
\end{theorem}

\begin{proof}
Assume (1). \ As we noted earlier, every graded-simple $L$-module $S$ is
graded isomorphic to $Lv/N$ where $v$ is a vertex and $N$ is a graded
submodule of $Lv$. Now $L$ and hence $Lv$ is a graded-projective $L$-module.
So if $S$ is finitely presented, then by Schanuel's Lemma, $N$ becomes a
finitely generated graded left ideal of $L$. By Theorem
\ref{Theorem of Hazrat}, $L$ is graded von Neumann regular and so, by
$(\ast\ast)$, $N$ is a direct summand of $L$ and that of $Lv$. This implies
that $S$ is graded-projective. Since every graded-simple $L$-module is
graded-projective, we appeal to Lemma \ref{All simple projective} to conclude
that $L$ is a graded semi-simple algebra. This proves (2).

Assume (2). Since every graded submodule of a graded semi-simple module is a
direct summand and since $L$ is graded projective, every graded homomorphic
image of $L$ and, in particular, every graded-simple $L$-module $S$ is
projective. Clearly $S$ is finitely presented. This proves (1).

Assume (1) and (2). By Lemma \ref{Infinite emitter not fp}, $E$ must be
row-finite. Also (2) implies $L=Soc^{gr}(L)$. From Theorem \ref{Graded Socle},
we conclude that $E^{0}$ is the saturated closure of the set of all line
points and Laurent vertices in $E$. This means (see e.g. Lemma 1.4,
\cite{ARS1}) that for each vertex $u$ there is an integer $n\geq0$ such that
every path beginning with $u$ and having length $\geq n$ ends at a line point
or a Laurent vertex, thus proving (3).

Assume (3). Since $E$ is row-finite, the stated condition implies that $E^{0}$
is the saturated closure of the set of all line points and Laurent vertices in
$E$. By Theorem \ref{Graded Socle}, $L=Soc^{gr}(L)$ is a direct sum of minimal
graded left ideals. This proves (2).
\end{proof}

We next consider when the Chen simple module $V_{[p]}$ corresponding to an
infinite path $p$ is finitely presented.

\ As mentioned in the introductory paragraphs of this section, we now use
Proposition \ref{V_[p] graded when p irrational}, to present a\ short and
simpler proof of a more general\ version of Theorem 2.20 of \cite{AMT} without
requiring the additional condition that no vertex on $p$ is an infinite emitter.

\begin{theorem}
\label{V_p finitely presented} \ Let $E$ be an arbitrary graph, $L=L_{K}(E)$
and let $p$ be an infinite irrational path in $E$. Then the following
properties are equivalent:

(i) \ \ The Chen simple left $L$-module $V_{[p]}$ is finitely presented;

(ii) \ $V_{[p]}$ is a projective left $L$-module;

(iii) The path $p$ contains a vertex which is a line point.
\end{theorem}

\begin{proof}
(i) =
$>$
(ii). As noted in Proposition \ref{V_[p] graded when p irrational}, $V_{[p]}$
is a graded-simple $L$-module in which $p$ is a homogeneous element. So if
$v=s(p)$, we then get a graded exact sequence
\[
0\rightarrow N\rightarrow Lv\overset{\theta}{\rightarrow}V_{[p]}\rightarrow0
\]
where $\theta(av)=av\cdot p$. If $V_{[p]}$ is finitely presented, then $N$
will be finitely generated. Since, by Theorem \ref{Theorem of Hazrat}, $L$ is
graded von Neumann regular, statement (**) above implies that $N$ will then be
a direct summand of $L$. Thus the above graded exact sequence splits and we
conclude that $V_{[p]}$ is graded-projective.

(ii) =%
$>$
(iii). If \ $V_{[p]}$ is a projective left $L$-module, then, being simple,
$V_{[p]}$ is isomorphic to a minimal left ideal $Lu$ where $u$ is a line
point. By Proposition 4.3 of \cite{C}, $p$ contains a a vertex which is a line
point, thus proving (iii).

(iii) =%
$>$
(i). Again if $p$ contains a vertex $u$ which is a line point, then
$V_{[p]}\cong Lu$, by (\cite{C}, Proposition 4.3). Since $Lu$ is projective,
$V_{[p]}$ is then clearly finitely presented.
\end{proof}

\section{Leavitt Path Algebras over which every simple module is graded}

In this short section, we characterize the Leavitt path algebras of arbitrary
graphs over which every simple left module is graded. \ A well-known theorem
(see \cite{T}) states that every two-sided ideal of a Leavitt path algebra
$L_{K}(E)$ is graded if and only if the graph $E$ satisfies Condition (K). A
natural question is to investigate when every one-sided ideal of $L$ is
graded. We show that this happens exactly when the graph $E$ contains no
cycles. A useful role is played by Theorems \ref{Gene-Rnga Theorem} and
\ref{V_p finitely presented}.

\begin{theorem}
Let $E$ be an arbitrary graph. Then for the Leavitt path algebra $L:=L_{K}(E)$
the following are equivalent:

(1) Every left/right ideal of $L$ is graded;

(2) Every simple left/right $L$-module is graded;

(3) The graph $E$ is acyclic.
\end{theorem}

\begin{proof}
(1) =%
$>$
(2) Suppose $S$ is a simple $L$-module. Since $L$ is a ring with local units,
there is a maximal left ideal $M$ such that $S\cong L/M$. Now, by assumption,
$M$ is a graded left ideal. Consequently, $L/M$ and hence $S$ can be made a
graded $L$-module.

(2) =%
$>$
(3). Now $E$ must be acyclic, because otherwise $E$ will contain a cycle $c$
and corresponding to the infinite rational path $p=ccc\cdot\cdot\cdot$, the
Chen simple module $V_{[p]}$ is, by Proposition
\ref{V_[p] graded when p irrational}, not a graded $L$-submodule, a
contradiction. Thus $E$ is acyclic.

(3) =%
$>$
(1). Suppose $E$ is acyclic. Now, by Theorem \ref{Gene-Rnga Theorem}, \ $L$ is
a directed union of graded subalgebras $B_{\lambda}$ where $\lambda\in I$, an
index set and where each $B_{\lambda}$ is a finite direct sum of copies of $K$
and a graded homomorphic image of a Leavitt path algebra of a finite acyclic
graph. By (Theorem 4.14, \cite{H-1}), Leavitt path algebras of finite acyclic
graphs are semi-simple algebras which have elementary gradings, that is, all
the matrix units are homogeneous. Consequently, every left ideal of each
$B_{\lambda}$ is graded. Since the $B_{\lambda}$ are graded subalgebras, each
$B_{\lambda}=%
{\displaystyle\bigoplus\limits_{n\in\mathbb{Z}}}
(B_{\lambda}\cap L_{n})$ where $L=%
{\displaystyle\bigoplus\limits_{n\in\mathbb{Z}}}
L_{n}$ is the $%
\mathbb{Z}
$-graded decomposition of $L$. Let $M$ be a left ideal of $L.$ To show that
$M$ is graded, we need only to show that $M=%
{\displaystyle\bigoplus\limits_{n\in\mathbb{Z}}}
(M\cap L_{n})$. Let $a\in M$. Then, for some $\lambda$, $a\in M\cap
B_{\lambda}$. Note that $M\cap B_{\lambda}=B_{\lambda}$ or a left ideal of
$B_{\lambda}$. Since every left ideal of $B_{\lambda}$ and in particular
$M\cap B_{\lambda}$ is graded, we can write $a=a_{n_{1}}+\cdot\cdot
\cdot+a_{n_{k}}$ where \
\[
a_{n_{i}}\in(M\cap B_{\lambda})\cap(B_{\lambda}\cap L_{n_{i}})\subseteq M\cap
L_{n_{i}}%
\]
for $i=1,\cdot\cdot\cdot,k$. This shows that $M=%
{\displaystyle\bigoplus\limits_{n\in\mathbb{Z}}}
(M\cap L_{n})$ and hence $M$ is a graded left ideal of $L$
\end{proof}

\section{ Graded self-injective Leavitt path algebras}

It was shown in Theorem \ref{Theorem of Hazrat} that every Leavitt path
algebra $L$ is always graded von Neumann regular. On the other hand, a
self-injective Leavitt path algebra is always von Neumann regular. So we wish
to describe the subclass of graded self-injective Leavitt path algebras. \ The
(non-graded) self-injective Leavitt path algebras were described in \cite{ARS}
as those associated to acyclic graphs. However, since the forgetful functor
$U:GrA\rightarrow\operatorname{mod}A$ is, in general, only a left-adjoint,
this implies that the graded projective modules are projective, but the graded
injective modules are not necessarily injective. In Theorem
\ref{Graded-injective LPAs}, we determine the class of graded self-injective
Leavitt path algebras. These algebras are shown to coincide with their graded
socle and thus are direct sums of matrix rings of arbitrary size over the
field $K$ and/or the ring $K[x^{t_{i}},x^{-t_{i}}]$ endowed with a natural
graded structure, where the $t_{i}$ are integers.

We begin with the following Lemma which is folklore. This will be used in
Lemma \ref{Large Goldie dimension}. For the sake of completeness, we outline
the proof.

\begin{lemma}
\label{Direct Product of K}Let $\tau$ be an infinite cardinal and let $K$ be
any field. Then the $K$-dimension of the direct product $P=%
{\displaystyle\prod\limits_{\tau}}
K$ of $\tau$ copies of $K$ is $\geq2^{\tau}$.
\end{lemma}

\begin{proof}
Clearly $|P|=\dim_{K}(P)\cdot|K|$. So if $|K|\leq\tau$, then $\dim
_{K}(P)=|P|=2^{\tau}$. Suppose $|K|>\tau$. Let $Q$ be the prime subfield of
$K$, so $|Q|$ is finite or countably infinite. If $P^{\prime}=%
{\displaystyle\prod\limits_{\tau}}
Q$, then, as noted earlier, $\dim_{Q}(P^{\prime})=2^{\tau}$. Let $\bar
{P}=K\otimes_{Q}P^{\prime}$. Then $\dim_{K}(\bar{P})=2^{\tau}$. Since $\bar
{P}\subseteq P$, $\dim_{K}(P)\geq2^{\tau}$.
\end{proof}

A \ useful tool in describing the graded self-injective Leavitt path algebras
is the following.

\begin{lemma}
\label{Large Goldie dimension}Suppose $v\in E^{0}$ such that $vLv$ has
$K$-dimension $\sigma$. If $\{Lx_{i}:i\in I\}$ is an independent family of
graded submodules of $Lv$ and $|I|=\tau\geq\max\{\sigma,\aleph_{0}\}$, then
$Lv$ is not graded injective as a left $L$-module.
\end{lemma}

\begin{proof}
Suppose, by way of contradiction, $Lv$ is a graded injective $L$-module. Let
$S=%
{\displaystyle\bigoplus\limits_{i\in I}}
Lx_{i}$. \ Then from the graded exact sequence $0\rightarrow S\rightarrow Lv$
we get the exact sequence%
\[
Hom_{gr}(Lv,Lv)\rightarrow Hom_{gr}(S,Lv)\rightarrow0.
\]
Now $Hom_{gr}(Lv,Lv)\cong vL_{0}v$ and
\[
Hom_{gr}(S,Lv)\supseteq Hom_{gr}(S,S)\supseteq%
{\displaystyle\prod\limits_{i\in I}}
Hom_{gr}(Lx_{i},Lx_{i})\supseteq%
{\displaystyle\prod\limits_{\tau}}
K.
\]
This leads to a contradiction since $vLv$ and hence $Hom_{gr}(S,Lv)$ has
$K$-dimension $\leq\sigma$ while, by Lemma \ref{Direct Product of K}, $%
{\displaystyle\prod\limits_{\tau}}
K$ has $K$-dimension $\geq2^{\tau}>\sigma$. Hence $Lv$ cannot be graded injective.
\end{proof}

The proof of the next Lemma is essentially due Pere Ara. We gratefully thank
him for offering such a short proof in place of our earlier elongated proof.
In the proof we shall be using the following result which is well-known in the
category of modules over an arbitrary ring $R$: If $M$ is an injective left
$R$-module, then $S/J(S)$ is self-injective, where $S$ is the endomorphism
ring of $M$ and $J(S)$ is the Jacobson radical of $S$ (see e.g. Theorem 13.1
in \cite{L}). The same proof works for the graded version of this theorem.

\begin{lemma}
\label{Lv injective => v finite emitter}If $v\in E^{0}$ and $Lv$ is a graded
injective left $L$-module, then $v$ cannot be an infinite emitter.
\end{lemma}

\begin{proof}
Suppose, by way of contradiction, the vertex $v$ is an infinite emitter. Let
$X=\{e_{n}:n=1,2,\cdot\cdot\cdot\}$ be a countably infinite subset of edges
emitted by $v$. Recall that the zero degree component $L_{0}$ of $L$ is von
Neumann regular (see Theorem 5.3, \cite{AMP}). Consequently, $vL_{0}v$ is von
Neumann regular ring with identity and clearly its Jacobson radical
$J(vL_{0}v)=0$. Since $Lv$ is graded injective, $vL_{0}v\cong End_{L}%
^{gr}(Lv)$ is a graded self-injective ring with identity (from the graded
version of the proof of Theorem 13.1, \cite{L}).

Let $S=%
{\displaystyle\bigoplus\limits_{i=1}^{\infty}}
vL_{0}ve_{i}e_{i}^{\ast}\subset vLv$. Let $\pi:S\rightarrow%
{\displaystyle\bigoplus\limits_{i=1}^{\infty}}
vL_{0}ve_{2i}e_{2i}^{\ast}$ be the coordinate projection map so $\pi
(e_{2i}e_{2i}^{\ast})=e_{2i}e_{2i}^{\ast}$ for all $i$ with $Ker(\pi)=%
{\displaystyle\bigoplus\limits_{i=1}^{\infty}}
vL_{0}ve_{2i+1}e_{2i+1}^{\ast}$. Applying the Baer's criterion to the graded
injective left module $vL_{0}v$, there is an element $a\in vL_{0}v$ (of degree
$0$) such that $\pi(x)=xa$. Let
\[
a=%
{\displaystyle\sum\limits_{j=1}^{m}}
k_{j}p_{j}q_{j}^{\ast}\qquad\qquad(\#)
\]
where $k_{j}\in K$, the $p_{j},q_{j}$ are paths with $|p_{j}|=|q_{j}|$ and
where we can assume that the monomials $p_{j}q_{j}^{\ast}$ are all different.
First notice that it is not possible for $a=v$, since $a$ satisfies the
condition $e_{1}e_{1}^{\ast}a=0$. Also, we can assume, since $vav=a$, that
$k_{j}vp_{j}q_{j}^{\ast}v=k_{j}p_{j}q_{j}^{\ast}$ for all $j=1,\cdot\cdot
\cdot,m$. Let $T=\{e_{n}\in X:e_{n}^{\ast}p_{j}\neq0$ for some $p_{j}\}$.
Clearly $T$ is a finite set. We distinguish two cases.

Case 1: Suppose there is a term in $(\#)$ which is a multiple of $v$, say
$k_{i}v$ with $k_{i}\neq0$. Note that, by our assumption, no other term in
$(\#)$ will then be a scalar multiple of $v$. As a consequence, for any $j\neq
i$, the $j$-th term in $(\#)$ cannot be of the form $k_{j}q_{j}^{\ast}$
(without the real path $p_{j}$) or of the form $k_{j}p_{j}$ (without the ghost
path $q_{j}^{\ast}$). Because, otherwise $\deg(q_{j}^{\ast})=0$ or $\deg
(p_{j})=0$ so the $j$-th term will be of the form $k_{j}v$ for $j\neq i$, and
this is impossible by our assumption.

Now for each odd integer $2n+1$, we have
\[
0=e_{2n+1}e_{2n+1}^{\ast}a=e_{2n+1}e_{2n+1}^{\ast}k_{i}v+e_{2n+1}%
e_{2n+1}^{\ast}%
{\displaystyle\sum\limits_{j=1,j\neq i}^{m}}
k_{j}p_{j}q_{j}^{\ast}.
\]
If we choose $e_{2n+1}\notin T$, then $e_{2n+1}^{\ast}p_{j}=0$ for all $j$ and
from the above equation we get,
\end{proof}

$-k_{i}e_{2n+1}e_{2n+1}^{\ast}=%
{\displaystyle\sum\limits_{j=1,j\neq i}^{m}}
k_{j}e_{2n+1}e_{2n+1}^{\ast}p_{j}q_{j}^{\ast}=0$, a contradiction.

Case 2: Suppose, for each $j$, the $j$-th term in $(\#)$ contains the real
path $p_{j}$. Then for any $e_{2n}\notin T$, we have $e_{2n}e_{2n}^{\ast}%
=\pi(e_{2n}e_{2n}^{\ast})=e_{2n}e_{2n}^{\ast}a=%
{\displaystyle\sum\limits_{j=1}^{m}}
k_{j}e_{2n}e_{2n}^{\ast}p_{j}q_{j}^{\ast}=0$, again a contradiction. Hence $v$
cannot be an infinite emitter.

\begin{proposition}
\label{Injective implies row-finite} If $L:=L_{K}(E)$ is graded left
self-injective, then (i) the graph $E$ is row-finite and, (ii) for any vertex
$v$, $Lv$ cannot have infinitely many independent graded submodules.
\end{proposition}

\begin{proof}
\textit{(i)} If $v$ is any vertex in  $E$, then the direct summand $Lv$ is an
injective left $L$-module. By Lemma \ref{Lv injective => v finite emitter},
$v$ cannot be an infinite emitter. Hence the graph $E$ must be row-finite. 
\end{proof}

\textit{(ii)} To prove the second statement, let $v\in E^{0}$. Since $E$ is
row-finite, there are only finitely many paths of a fixed length $n$ starting
from $v$ and\ so the set $Y$ of all finite paths $\alpha$ in $E$ with
$s(\alpha)=v$ is at most countable. Now every element of $vLv$ is a finite sum
of the form $%
{\displaystyle\sum\limits_{i=1}^{n}}
k_{i}p_{i}q_{i}^{\ast}$ where $k_{i}\in K$, $p_{i},q_{i}$ are finite paths
with $s(p_{i})=s(q_{i})=v$ and $r(p_{i})=r(q_{i})$. Since the cardinality of
the set of all finite subsets of the countable set $Y$ is again countable, we
conclude that the $K$-dimension of $vLv$ is $\leq\aleph_{0}$. Since $Lv$ is a
graded injective $L$-module, we appeal to Lemma \ref{Large Goldie dimension}
to conclude that $Lv$ cannot have infinitely many independent graded submodules.

\begin{corollary}
\label{NE condition}If $L=L_{K}(E)$ is graded self-injective, then no cycle in
$E$ will have an exit
\end{corollary}

\begin{proof}
Suppose, by way of contradiction, $E$ contains a cycle $c$ based at a vertex
$v$ and having an exit $f$ at $v$. We modify the ideas in the proof of (ii) =
$>$
(iii) of Theorem 3.9 in \cite{AAPS}. Now for any $n$, $v-c^{n}(c^{\ast})^{n}$
is a homogeneous idempotent and $v-c^{n}(c^{\ast})^{n}=(v-c^{n}(c^{\ast}%
)^{n})(v-c^{n+1}(c^{\ast})^{n+1})$. Consequently we get an ascending chain of
graded ideals contained in $Lv:$
\[
L(v-cc^{\ast})\subseteq L(v-c^{2}(c^{\ast})^{2}\subseteq\cdot\cdot\cdot
\qquad\qquad\qquad(\#\#)
\]
We claim that the containments in the chain $(\#\#)$ are strict. Suppose, on
the contrary, $L(v-c^{n}(c^{\ast})^{n})=L(v-c^{n+1}(c^{\ast})^{n+1})$ for some
$n$. Then $(v-c^{n+1}(c^{\ast})^{n+1})=a(v-c^{n}(c^{\ast})^{n})$ for some
$a\in L$. \ Multiplying the last equation by $c^{n}$ on the right, we get
$c^{n}-c^{n+1}c^{\ast}=0$ and so $c^{n}=c^{n+1}c^{\ast}$. Again multiplying
the last equation on the right by $f$, we get $c^{n}f=0$, a contradiction.
\ Thus $(\#\#)$ is a strictly increasing infinite chain\ of direct summands of
$Lv$. Let, for each $n\geq1$, $L(v-c^{n+1}(c^{\ast})^{n+1})=L(v-c^{n}(c^{\ast
})^{n})\oplus B_{n}$. Then%
\[
A=%
{\displaystyle\bigcup\limits_{n\geq1}}
L(v-c^{n}(c^{\ast})^{n})=L(v-cc^{\ast})\oplus%
{\displaystyle\bigoplus\limits_{n=1}^{\infty}}
B_{n}\subset Lv.
\]
\ Since $Lv$ is graded-injective, this is a contradiction by Proposition
\ref{Injective implies row-finite}. Thus no cycle in $E$ has an exit.
\end{proof}

\begin{proposition}
\label{Injective implies paths contain line or Laurent vertex}If $L$ is graded
left self-injective, then every infinite path in $E$ contains a line point or
\ is tail-equivalent to a rational path induced by a cycle without exits and
thus contains a Laurent vertex.
\end{proposition}

\begin{proof}
Let $\mu$ be a an infinite path in $E$ with $s(\mu)=v$. Clearly $Lv$ is a
graded injective $L$-module. Suppose $\mu$ does not contain a line point. Then
$\mu$ either contains a cycle or contains infinitely many bifurcating
vertices. Suppose $\mu$ contains infinitely many bifurcating vertices $v_{n}$
and bifurcating edges $f_{n}$ \ with $s(f_{n})=v_{n}$. Without loss of
generality, we may assume $v=s(\mu)=v_{1}$. For each $n\geq1$, let $\gamma
_{n}$ be the part of $\mu$ with $r(\gamma_{n})=v_{n+1}$ and $s(\gamma
_{n})=v_{n}$ and thus $\mu$ is the concatenation of paths, $\mu=\gamma
_{1}\gamma_{2}\cdot\cdot\cdot\gamma_{n}\cdot\cdot\cdot$. Let, for each $n$,
$\epsilon_{n}=\gamma_{1}\gamma_{2}\cdot\cdot\cdot\gamma_{n}f_{n+1}%
f_{n+1}^{\ast}\gamma_{n}^{\ast}\cdot\cdot\cdot\gamma_{1}^{\ast}$. Then the set
$\{\epsilon_{n}:n\geq1\}$ is a set of homogeneous orthogonal idempotents
giving rise to an infinite set of independent graded submodules of $Lv$ and
this, by Proposition \ref{Injective implies row-finite}(ii), contradicts the
graded injectivity of $Lv$. So $\mu$ must contain a cycle $c$ which, by
Corollary \ref{NE condition}, has no exits. This means the infinite path $\mu$
must be tail-equivalent to the infinite rational path \ $ccc\cdot\cdot\cdot$
\ \ \ and contains the Laurent vertex $v$, where $v$ is the base of the cycle
$c$.
\end{proof}

\begin{theorem}
\label{Graded-injective LPAs}Let \ $E$ be an arbitrary graph. Then the
following are equivalent for $L:=L_{K}(E)$:

(1) $L$ is \ graded left/right self-injective;

(2) $E$ is row-finite, no cycle in $E$ has an exit and every infinite path in
$E$ is tail-equivalent to either straight line segment containing a line point
or to a rational path $ccc\cdot\cdot\cdot$ for some cycle $c$ (without exits);

(3) There is a graded isomorphism%
\[
L_{K}(E)\cong_{gr}%
{\displaystyle\bigoplus\limits_{v_{i}\in X}}
M_{\Lambda_{i}}(K)((|\overset{-}{p^{v_{i}}|)}\oplus%
{\displaystyle\bigoplus\limits_{w_{j}\in Y}}
M_{\Upsilon_{j}}(K[x^{t_{j}},x^{-t_{j}}])(|\overset{-}{q^{w_{j}}}|)
\]
where $\Lambda_{i}$,$\Upsilon_{j}$ are suitable index sets, the $t_{j}$ are
positive integers, .\ $X$ is the set of representatives of distinct
equivalence classes of line points in $E$ and $Y$ is the set of all distinct
cycles (without exits) in $E$.
\end{theorem}

\begin{proof}
Now (1) =%
$>$
(2) by Propositions \ref{Injective implies row-finite},
\ref{Injective implies paths contain line or Laurent vertex} and Corollary
\ref{NE condition}.

Assume (2). We claim that $E^{0}=H$, the hereditary saturated closure of the
set of all line points and Laurent vertices in $E$. \ Suppose, on the
contrary, there is a $v_{1}\in E^{0}\backslash H$. Then, as $E$ is row-finite,
there must be an edge $e_{1}$ with $s(e_{1})=v_{1}$ and $r(e_{1})=v_{2}\notin
H$. This then implies there is an edge $e_{2}$ with $s(e_{2})=v_{2}$ and
$r(e_{2})=v_{3}\notin H$. Proceeding like this we get an infinite path
$e_{1}e_{2}e_{3}\cdot\cdot\cdot$ \ which neither contains a line point nor a
cycle without exits (and thus no Laurent vertex), a contradiction to our
hypothesis. Hence $E^{0}=H$ and so $L$ is the ideal generated by $H$. Then (3)
follows from Theorem \ref{Graded Socle}.

Finally, (3) =%
$>$
(1) since $L$ is graded semi-simple and so is graded injective as a left/right
$L$-module .
\end{proof}

\end{document}